\documentclass{amsart}

\newtheorem{lemma}{Lemma}[section]
\newtheorem{theorem}[lemma]{Theorem}
\newtheorem{remark}[lemma]{Remark}

\newtheorem{coro}[lemma]{Corollary}
\newtheorem{definition}[lemma]{Definition}
\newtheorem{example}[lemma]{Example}

\allowdisplaybreaks

\title[Averaging Principle on Semi-axis for Semi-linear
\ldots Equations]{Averaging Principle on Semi-axis for Semi-linear
Differential Equations}

\author{David ~Cheban}
\address[D. Cheban]{%
State University of Moldova\\ Department of Mathematics\\ Laboratory of Fundamental and Applied Mathematics\\
A. Mateevich Street 60\\ MD--2009 Chi\c{s}in\u{a}u, Moldova}
\email[D. Cheban]{david.ceban@usm.md, davidcheban@yahoo.com}

\date{\today}

\date{\today}

\begin{document}
\begin{abstract}
{We establish an averaging principle on the real semi-axis for
semi-linear equation
\begin{equation}\label{eqAb1}
x'=\varepsilon (\mathcal A x+f(t)+F(t,x))\nonumber
\end{equation}
with unbounded closed linear operator $\mathcal A$ and
asymptotically Poisson stable (in particular, asymptotically
stationary, asymptotically periodic, asymptotically
quasi-periodic, asymptotically almost periodic, asymptotically
almost automorphic, asymptotically recurrent) coefficients. Under
some conditions we prove that there exists at least one solution,
which possesses the same asymptotically recurrence property as the
coefficients, in a small neighborhood of the stationary solution
to the averaged equation, and this solution converges to the
stationary solution of averaged equation uniformly on the real
semi-axis when the small parameter approaches to zero.}
\end{abstract}

\maketitle

\section{Introduction}\label{S1}

In the present paper, we establish an averaging principle on the
real semi-axis, i.e. the analogue of Bogolyubov's second theorem,
for semi-linear differential equations
\begin{equation}\label{eqI1}
x'=\varepsilon (Ax+f(t)+F(t,x)) .
\end{equation}
If there exists a stationary solution $\bar{\psi}$ for the
averaged equation
\begin{equation}\label{eqI2}
x'= Ax+\bar{f}+\bar{F}(x) ,\nonumber
\end{equation}
then there exists in a small neighborhood $\bar{\psi}$ a solution
of the original equation which is defined on the semi-axis
$\mathbb R_{+}$ and has the same recurrence property in limit as
the coefficients of the original equation. Note that this research
was stimulated by works \cite{Che_1988} and \cite{CL_2021}.

To be more precise, we investigate the semi-linear differential
equation (\ref{eqI1}) with asymptotically Poisson stable (in
particular, asymptotically stationary, asymptotically periodic,
asymptotically quasi-periodic, asymptotically Bohr almost
periodic, asymptotically almost automorphic, asymptotically
Birkhoff recurrent, asymptotically Levitan almost periodic,
asymptotically almost recurrent, asymptotically Poisson stable) in
time coefficients. Under some conditions, this equation has a
unique bounded on $\mathbb R_{+}$ solution $\psi_{\varepsilon}$
with $P_{-}\psi_{\varepsilon}(0)=0$ which has the same
asymptotically recurrent properties as the coefficients, see
\cite{Che_2009} for details.

The paper is organized as follows. In the second section we
collect some known notions and facts. Namely we present the
construction of shift dynamical systems, definitions and basic
properties of Poisson stable and asymptotically Poisson stable
functions, comparability method for asymptotically Poisson stable
motions, and the existence of compatible in the limit of solutions
for semi-linear differential equations. In the third and fourth
sections, we study the problem of existence of bounded solutions
on the semi-axis for linear and semi-linear differential
equations. Section fifth is dedicated to the study a special
classes of sectorial operators. Namely we establish some
properties of hyperbolic sectorial operators. Finally, in the
sixth section the problem of averaging on semi-axis for
semi-linear differential equations are studied.

\section{Preliminaries}\label{S2}

Let $(X,\rho)$ be a complete metric space. Let $\mathbb R$
(respectively, $\mathbb R_{+}$) be the set of all real
(respectively, real non-negative) numbers and $\mathbb T =\mathbb
R$ or $\mathbb R_{+}$. Denote by $C(\mathbb T,X)$ the space of all
continuous functions $\varphi :\mathbb T \to X$ equipped with the
compact-open topology. This topology can be defined by the
following distance
\begin{equation*}\label{eqD_1}
d(\varphi,\psi):=\sup\limits_{L>0}\min\{\max\limits_{|t|\le
L}\rho(\varphi(t),\psi(t)),L^{-1}\}.
\end{equation*}
The space $(C(\mathbb T,X),d)$ is a complete metric space (see,
for example, \cite[ChI]{Sch72},\cite{Sch85,sib}).

\begin{lemma}\label{l1} {\rm(\cite[ChI]{Sch72},\cite{Sch85,sib})} The following statements
hold:
\begin{enumerate}
\item $d(\varphi,\psi) = \varepsilon$ if and only if
$$
\max\limits_{|t|\le
\varepsilon^{-1}}\rho(\varphi(t),\psi(t))=\varepsilon ;
$$
\item $d(\varphi,\psi)<\varepsilon$ if and only if
$$
\max\limits_{|t|\le
\varepsilon^{-1}}\rho(\varphi(t),\psi(t))<\varepsilon ;
$$
\item $d(\varphi,\psi)>\varepsilon$ if and only if
$$
\max\limits_{|t|\le
\varepsilon^{-1}}\rho(\varphi(t),\psi(t))>\varepsilon .
$$
\end{enumerate}
\end{lemma}

\begin{remark}\label{remD1} \rm
1. The distance $d$ generates on $C(\mathbb T,X)$ the compact-open
topology.

2. The following statements are equivalent:
\begin{enumerate}
\item $d(\varphi_n,\varphi)\to 0$ as $ n\to \infty$; \item
$\lim\limits_{n\to \infty}\max\limits_{|t|\le
L}\rho(\varphi_n(t),\varphi(t))=0$ for every $L>0$; \item there
exists a sequence $l_n\to +\infty$ such that $\lim\limits_{n\to
\infty}\max\limits_{|t|\le l_n}\rho(\varphi_n(t),\varphi(t))=0$.
\end{enumerate}
\end{remark}

Let us recall the types of Poisson stable functions
\cite{Sch72,Sch85,Sel,sib} (see also \cite[Ch.I]{Che_2020}).

\begin{definition} \rm
A function $\varphi\in C(\mathbb T,X)$ is called stationary
(respectively, $\tau$-periodic) if $\varphi(t)=\varphi(0)$
(respectively, $\varphi(t+\tau)=\varphi(t)$) for all $t\in \mathbb
T$.
\end{definition}

\begin{definition} \rm
Let $\varepsilon >0$. A number $\tau \in \mathbb T$ is called
$\varepsilon$-almost period of $\varphi$ if
$\rho(\varphi(t+\tau),\varphi(t))<\varepsilon$ for all
$t\in\mathbb T$. Denote by $\mathcal T(\varphi,\varepsilon)$ the
set of $\varepsilon$-almost periods of $\varphi$.
\end{definition}

Recall that a subset $A\subset \mathbb T$ is said to be relatively
dense if for each $\varepsilon >0$ there exists $l\in \mathbb T,\
l=l(\varepsilon)>0$ such that
$$
\mathcal
T(\varphi,\varepsilon)\cap [a,a+l]\not=\emptyset
$$
for all
$a\in\mathbb T$, where $[a,a+l]:=\{t\in \mathbb T|\ a\le t\le
a+l\}$.

\begin{definition} \rm
A function $\varphi \in C(\mathbb T,X)$ is said to be Bohr almost
periodic if the set of $\varepsilon$-almost periods of $\varphi$
is relatively dense for each $\varepsilon >0$.
\end{definition}


\begin{definition}\rm
For given $\varphi\in C(\mathbb T, X)$ and $h\in \mathbb T$,
denote by $\varphi^h$ the $h$-translation of $\varphi$, i.e.,
$\varphi^h(t)=\varphi(h+t)$ for $t\in\mathbb T$. Denote by
$H(\varphi)$ the hull of $\varphi$, i.e.,
\[
H(\varphi):=\{\psi\in C(\mathbb R, X): \psi=\lim_{n\to\infty}
\varphi^{h_n} \hbox{ for some sequence } \{h_n\} \subset \mathbb
R\}.
\]
\end{definition}

It is well-known (see, e.g. \cite[Ch.I]{Che_2015}) that the
mapping $\sigma: \mathbb R\times C(\mathbb T, X)\to C(\mathbb T,
X)$ defined by $\sigma(h,\varphi) =\varphi^h$ possesses the
following properties:
\begin{enumerate}
\item $\sigma(0,\varphi)=\varphi $; \item
$\sigma(h_1+h_2,\varphi)=\sigma(h_2,\sigma(h_1,\varphi))$; \item
the mapping $\sigma$ is continuous.
\end{enumerate}
Thus the triplet $(C(\mathbb T,X),\mathbb T,\sigma)$ is a
dynamical system (shift or Bebutov's dynamical system) on the
space $C(\mathbb T,X)$.

\begin{definition}\label{defLS1} A function $\varphi \in C(\mathbb
T,X)$is said to be Lagrange stable if
$\Sigma_{\varphi}:=\{\sigma(\tau,\varphi)|\ \tau \in \mathbb T\}$
is pre-compact in the space $C(\mathbb T,X)$.
\end{definition}

\begin{lemma}\label{lLS1} (\cite[Ch.IV]{Bro79},
\cite[Ch.I]{Sch72}, \cite{SK_1974}) The function $\varphi \in
C(\mathbb T,X)$is Lagrange stable if and only if the following
conditions are fulfilled:
\begin{enumerate}
\item $\varphi(\mathbb T)$ is pre-compact in $X$; \item the
mapping $\varphi :\mathbb T\to X$ is uniformly continuous.
\end{enumerate}
\end{lemma}

\begin{definition} \rm
A number $\tau\in\mathbb T$ is said to be $\varepsilon$-shift for
$\varphi \in C(\mathbb T,X)$ if
$d(\varphi^{\tau},\varphi)<\varepsilon$.
\end{definition}

\begin{definition} \rm
A function $\varphi \in C(\mathbb T,X)$ is called {\em almost
recurrent} if for every $\varepsilon >0$ the set $\{\tau :\
d(\varphi^{\tau},\varphi)<\varepsilon\}$ is relatively dense.
\end{definition}

\begin{definition} \rm
A function $\varphi \in C(\mathbb T,X)$ is called recurrent if it
is almost recurrent and Lagrange stable.
\end{definition}


Let $Y$ be a complete metric space.

\begin{definition} \rm
A function $\varphi\in C(\mathbb T,X)$ is called Levitan almost
periodic if there exists a Bohr almost periodic function $\psi \in
C(\mathbb T,Y)$ such that for any $\varepsilon >0$ there exists
$\delta =\delta (\varepsilon)>0$ such that
$d(\varphi^{\tau},\varphi)<\varepsilon$ for all $\tau \in \mathcal
T(\psi,\delta)$, recalling that $\mathcal T(\psi,\delta)$ denotes
the set of $\delta$-almost periods of $\psi$.
\end{definition}

\begin{remark} \rm
\begin{enumerate}
\item  Every Bohr almost periodic function is Levitan almost
periodic.

\item  The function $\varphi \in C(\mathbb T,\mathbb R)$ defined
by equality $\varphi(t)=\dfrac{1}{2+\cos t +\cos \sqrt{2}t}$ is
Levitan almost periodic, but it is not Bohr almost periodic
\cite[ChIV]{Lev-Zhi}.
\end{enumerate}
\end{remark}

\begin{definition}\label{defAA01} \rm
A function $\varphi \in C(\mathbb R,X)$ is said to be almost
automorphic if it is Levitan almost periodic and Lagrange stable.
\end{definition}

\begin{definition} \rm
A function $\varphi \in C(\mathbb T,X)$ is called quasi-periodic
with the spectrum of frequencies $\nu_1,\nu_2,\ldots,\nu_k$ if the
following conditions are fulfilled:
\begin{enumerate}
\item the numbers $\nu_1,\nu_2,\ldots,\nu_k$ are rationally
independent; \item there exists a continuous function $\Phi
:\mathbb T^{k}\to X$ such that
$\Phi(t_1+2\pi,t_2+2\pi,\ldots,t_k+2\pi)=\Phi(t_1,t_2,\ldots,t_k)$
for all $(t_1,t_2,\ldots,t_k)\in \mathbb T^{k}$; \item
$\varphi(t)=\Phi(\nu_1 t,\nu_2 t,\ldots,\nu_k t)$ for $t\in
\mathbb T$.
\end{enumerate}
\end{definition}

Let $\varphi \in C(\mathbb T,X)$. Denote by $\mathfrak
L_{\varphi}$ the family of all sequences $\{t_n\}\subset \mathbb
T$ such that $t_n\to +\infty$ and $\{\varphi^{t_n}\}$ converges in
$C(\mathbb T,X)$ as $n\to \infty$.

\begin{definition} \rm
A function $\varphi \in C(\mathbb T,X)$ is said to be comparable
by character of recurrence in the limit \cite[Ch.II]{Che_2009} (or
shortly comparable in the limit) with the function $\psi \in
C(\mathbb T,Y)$ if $\mathfrak L_{\psi}\subseteq \mathfrak
L_{\varphi}$.
\end{definition}

\begin{definition}\label{defAP01} A function $\varphi \in C(\mathbb
T,X)$is said to be asymptotically stationary (respectively,
asymptotically $\tau$-periodic, asymptotically Levitan almost
periodic, asymptotically quasi-periodic with the spectrum of
frequencies $\nu_1,\nu_2,\dots,\nu_k$, asymptotically Bohr almost
periodic, asymptotically Bohr almost automorphic, asymptotically
Birkhoff recurrent) \cite[Ch.I]{Che_2009} if there exists a
stationary (respectively, $\tau$-periodic, Levitan almost
periodic, quasi-periodic with the spectrum of frequencies
$\nu_1,\nu_2,\dots,\nu_k$, Bohr almost periodic, Bohr almost
automorphic, Birkhoff recurrent) function $p\in C(\mathbb T,X)$
such that $\lim\limits_{t\to +\infty}\rho(\varphi(t),p(t))=0$.
\end{definition}

\begin{theorem}\label{th2} (\cite[ChII]{Che_2009}, \cite[Ch.I]{Che_2020})
Let $\varphi \in C(\mathbb T,X)$
be comparable in the limit by character of recurrence with
$\psi\in C(\mathbb T,Y)$. If the function $\psi$ is asymptotically
stationary (respectively, asymptotically $\tau$-periodic,
asymptotically quasi-periodic with the spectrum of frequencies
$\nu_1,\nu_2,\dots,\nu_k$, asymptotically Bohr almost periodic,
asymptotically Bohr almost automorphic, asymptotically Birkhoff
recurrent, positively Lagrange stable), then so is $\varphi$.
\end{theorem}

Let $W\subseteq X$ be a bounded (respectively, compact) subset of
$X$. Denote by $C(\mathbb T\times W,X)$ the space of all
continuous functions $f :\mathbb T\times W \to X$ equipped with
the compact-open topology. On the space $C(\mathbb \mathbb T\times
W,X)$ is defined \cite[Ch.I]{Che_2015} a shift dynamical system
(Bebutov's dynamical system) $(C(\mathbb \mathbb T\times
W,X),\mathbb T,\sigma)$, where $\sigma :\mathbb T\times C(\mathbb
T\times W,X)\to C(\mathbb T\times W,X)$ is defined by equality
$\sigma(\tau,f)=f^{\tau}$ and $f^{\tau}(t,x)=f(t+\tau,x)$ for any
$f\in C(\mathbb T\times W,X)$, $t,\tau \in \mathbb T$ and $x\in
W$.

If $Q$ is a compact subset of $W$, then the topology on the space
$C(T\times Q,X)$ can be \cite{shcher67} defined by distance
\begin{equation}\label{eqD1}
d(f,g):=\sup\limits_{l>0}\min\{\max\limits_{|t|\le l,\ x\in
Q}\rho(f(t,x),g(t,x)), l^{-1}\} .\nonumber
\end{equation}

\begin{definition}\label{defAA1} A function $f\in C(\mathbb
T\times W,X)$ is said to be:
\begin{enumerate}
\item[--] Bohr almost periodic in $t\in \mathbb T$ uniformly with
respect to $x\in W$ if for arbitrary positive number $\varepsilon$
the set
\begin{equation}\label{eqAA2}
\mathcal T(\varepsilon,f):=\{\tau \in \mathbb T|\
\sup\limits_{t\in\mathbb T,\ x\in
W}\rho(f(t+\tau,x),f(t,x))<\varepsilon\}\nonumber
\end{equation}
is relatively dense; \item[--] asymptotically Bohr almost periodic
in $t\in \mathbb T$ uniformly with respect to $x\in W$ if there
exists a Bohr almost periodic in $t\in \mathbb T$ uniformly with
respect to $x\in W$ function $P\in C(\mathbb T\times W,X)$ such
that
\begin{equation}\label{eqAA3}
\lim\limits_{t\to +\infty}\sup\limits_{x\in
W}\rho(f(t,x),P(t,x))=0.\nonumber
\end{equation}
\end{enumerate}
\end{definition}

Analogically we can define a notion of
\begin{enumerate}
\item[--] Levitan almost periodicity (respectively, almost
automorphy, Birkhoff recurrence and so on) in $t\in \mathbb T$
uniformly with respect to $x\in W$; \item[--] asymptotically
Levitan almost periodicity (respectively, asymptotically almost
automorphy, asymptotically Birkhoff recurrence and so on) in $t\in
\mathbb T$ uniformly with respect to $x\in W$.
\end{enumerate}

If $Q$ is a compact subset of $W$ and $f\in C(\mathbb T\times
W,X)$ then we denote by $f_{Q}$ the restriction $f$ on $\mathbb
T\times Q$, i.e., $f_{Q}:=f\big{|}_{\mathbb T\times Q}$ and
$\mathfrak M_{f_{Q}}:=\{\{t_{n}\}\subset \mathbb T|\ f_{Q}^{t_n}$
converges in $C(\mathbb T\times Q,X)$ $\}$ (respectively,
$\mathcal L_{f_{Q}}:=\{\{t_n\}\in \mathfrak M_{f_{Q}}|\ $ such
that $t_n\to +\infty$ as $n\to \infty$ $\}$.

Note that $C(Q,X)$ is a complete metric space equipped with the
distance
\begin{equation}\label{eqD2}
d(F_1,F_2):=\max\limits_{x\in Q}\rho(F_1(x),F_2(x)).\nonumber
\end{equation}

Let $Q$ be a compact subset of $X$. Consider the functional spaces
$C(\mathbb T\times Q,X)$ and $C(\mathbb T,C(Q,X))$. For any $f\in
C(\mathbb T\times Q,X)$ corresponds \cite{shcher67} a unique map
$F\in C(\mathbb T,C(Q,X))$ defined by equality
\begin{equation}\label{eqAA4}
F_{f}(t):=f(t,\cdot). \nonumber
\end{equation}
This means that it is well defined the mapping $h:C(\mathbb
T\times Q,X)\to C(\mathbb T,C(Q,X))$ by equality
\begin{equation}\label{eqD_3}
h(f)=F_{f}.\nonumber
\end{equation}

The following statement holds.

\begin{theorem}\label{thAA1} \cite{shcher67} Let $Q$ be a compact subset of $X$,
then the mapping $h: C(\mathbb T\times Q,X)\to C(T,C(Q,X))$
possesses the following properties:
\begin{enumerate}
\item $h$ is a homeomorphism of the space $C(\mathbb T\times Q,X)$
onto $C(\mathbb T,C(Q,X))$; \item for any $l>0$ and $f,g\in
C(\mathbb T\times Q,X)$ we have
\begin{equation}\label{eqH_1}
\max\limits_{|t|\le l,\ x\in
Q}\rho(f(t,x),g(t,x))=\max\limits_{|t|\le
l}\rho(F_{f}(t),F_{g}(t));\nonumber
\end{equation}
\item $d(f,g)=d(F_{f},F_{g})$ for any $f,g\in C(\mathbb T\times
Q,X)$,
\end{enumerate}
i.e., the mapping $h$ defines an isometry between $C(T\times Q,X)$
and $C(\mathbb T,C(Q,X))$.
\end{theorem}

\begin{lemma}\label{lAA0} The following relation
\begin{equation}\label{eqH}
h(\sigma(\tau,f))=\sigma(\tau,h(f))
\end{equation}
holds for any $(\tau,f)\in \mathbb T\times C(\mathbb T\times
Q,X)$, i.e., $h$ is an isometric homeomorphism of dynamical system
$(C(\mathbb T\times Q,X),\mathbb T,\sigma)$ onto $(C(\mathbb
T,C(Q,X)),\mathbb T,\sigma)$.
\end{lemma}
\begin{proof} Note that
\begin{eqnarray}\label{eqH1}
& h(\sigma(\tau,f))=h(f^{\tau})=F_{f^{\tau}}, \nonumber \\
& F_{f^{\tau}}(t)=f^{\tau}(t,\cdot)=f(t+\tau,\cdot)
\end{eqnarray}
and
\begin{eqnarray}\label{eqH2}
& \sigma(\tau,h(f))=h(f)^{\tau}=(F_{f})^{\tau},\nonumber \\
& F_{f}^{\tau}(t)=F_{f}(t+\tau)=f(t+\tau,\cdot)
\end{eqnarray}
for any $(\tau,f)\in\mathbb T\times C(\mathbb T\times Q,X)$ and
$t\in \mathbb T$. From (\ref{eqH1})-(\ref{eqH2}) we obtain
(\ref{eqH}). Lemma is proved.
\end{proof}

\begin{definition}\label{defLS2} A function $F\in C(\mathbb T\times
Q,X)$ is said to be Lagrange stable if the set
$\Sigma_{F}:=\{\sigma(\tau,F)|\ \tau\in \mathbb T\}$ is
pre-compact in $C(\mathbb T\times Q,X)$.
\end{definition}

\begin{lemma}\label{lLS2} \cite[Ch.IV]{Bro79}, \cite[Ch.III]{Sch72}, \cite{SK_1974}
Assume that $Q$ is a compact subset of
$X$. A function $F\in C(\mathbb T\times Q,X)$ is Lagrange stable
if and only if the following conditions hold:
\begin{enumerate}
\item the set $F(\mathbb T\times Q)$ is pre-compact in $X$; \item
the mapping $F:\mathbb T\times Q\to X$ is uniformly continuous.
\end{enumerate}
\end{lemma}

\begin{lemma}\label{lLS3} \cite[Ch.III]{Sch72}, \cite{SK_1974}  Let $\varphi\in C(\mathbb
T,X)$be a Lagrange stable function and
$Q:=\overline{\{\varphi(\mathbb T)\}}$, where by bar is denoted
the closure in $X$. If the function $F\in C(\mathbb T\times Q,X)$
is Lagrange stable, then the function $\psi \in C(T,X)$ defined by
equality $\psi(t):=F(t,\varphi(t))$ for any $t\in\mathbb T$ is
also Lagrange stable.
\end{lemma}

\begin{theorem}\label{thAA2} Let $Q$ be a compact subset of $X$.
Then the following statements are equivalent: \begin{enumerate}
\item the function $f\in C(\mathbb T\times Q,X)$ is Bohr almost
periodic (respectively, Levitan almost periodic) in $t\in \mathbb
T$ uniformly w.r.t. $x\in Q$; \item the function $F:=h(f)\in
C(\mathbb T,C(Q,X))$ is Bohr almost periodic (respectively,
Levitan almost periodic).
\end{enumerate}
\end{theorem}
\begin{proof} This
statement follows directly from the corresponding definitions and
Theorem \ref{thAA1} and Lemma \ref{lAA0}.
\end{proof}

\begin{remark}\label{remR01} Statements analogous to Theorem \ref{thAA2} also
hold for other classes of functions that are somehow almost
automorphic, Birkhoff recurrent, etc.
\end{remark}

\section{Bounded on the semi-axis solutions of linear differential
equations}\label{S3}

Let $(\mathfrak B,|\cdot|)$ be a Banach space with the norm
$|\cdot|$, $\mathcal A: D(\mathcal A)\to \mathfrak B$ be a linear
operator acting from $D(\mathcal A)\subseteq \mathfrak B$ to
$\mathfrak B$ generating a $C_{0}$-semigroup $\{U(t)\}_{t\ge 0}$
on the space $\mathfrak B$. Denote by $B[0,r]:=\{x\in \mathfrak B
|\ |x|\le r \}$.

Consider a linear nonhomogeneous equation
\begin{equation}\label{eqLN1}
\dot{x}=Ax+f(t)
\end{equation}
on the space $\mathfrak B$, where $f\in C(\mathbb R_{+},\mathfrak
B)$ and $A$ is an infinitesimal generator which generates a
$C_0$-semigroup $\{U(t)\}_{t\ge 0}$ acting on $\mathfrak B$.

Recall that a continuous function $u: [0,a)\to \mathfrak B$
($a>0$) is called a weak (mild) solution of equation (\ref{eqLN1})
passing through the point $x\in \mathfrak B$ at the initial moment
$t=0$ if
\begin{equation}\label{eqIE1}
u(t)=U(t)x + \int_{0}^{t}U(t-s)f(s)ds \nonumber
\end{equation}
for all $t\in [0,T]$ and $0<T<a$.

\begin{definition} A mild solution $\varphi \in C(\mathbb R_{+},\mathfrak B)$ of equation (\ref{eqLN1}) is called
\cite[ChIII]{Sch72} compatible by character of recurrence in the
limit if it is comparable by character of recurrence in the limit
with $f$, i.e., if $\mathfrak L_{f}\subseteq \mathfrak
L_{\varphi}$.
\end{definition}

\begin{definition}\label{defH1} A semigroup of operators $\{U(t)\}_{t\ge
0}$ is said to be hyperbolic if there is a projection $\mathcal P$
and constants $\mathcal N,\nu >0$ such that each $U(t)$ commutes
with $\mathcal P$, $U(t): Im \mathcal Q\mapsto Im \mathcal Q$ is
invertible and for every $x\in E$
\begin{equation}\label{eqLSDE01}
|U(t)\mathcal P x|\le \mathcal N e^{-\nu t}|x|,\ \ \ \mbox{for}\
t\ge 0;\nonumber
\end{equation}
\begin{equation}\label{eqLSDE02}
|U_{\mathcal Q}(t) x|\le \mathcal N e^{\nu t}|x|,\ \ \ \mbox{for}\
t < 0;\nonumber
\end{equation}
where $\mathcal Q:=I-\mathcal P$ and, for $t<0$, $U_{\mathcal
Q}(t):=[U(-t)\mathcal Q]^{-1}$.
\end{definition}

A Green's function $G(t)$ (see, for example, \cite[Ch.VII]{C-L})
for hyperbolic semigroup $U(t)$ is defined by
\begin{equation}\label{eqLS2}
 G(t):=\left\{\begin{array}{ll}
&\!\! U(t)\mathcal P,\;\ \ \ \mbox{if}\;\ t\ge 0 \\[2mm]
&\!\! -U_{\mathcal Q}(t), \;\ \mbox{if}\;\ t<0\ .
\end{array}
\right.\nonumber
\end{equation}

\begin{remark}\label{remG1} Let $G$ be the Green's
function for hyperbolic semigroup $\{U(t)\}_{t\ge 0}$, then there
are positive numbers $\mathcal N$ and $\nu$ such that
\begin{equation}\label{eqHG1}
\|G(t)\|\le \mathcal N e^{-\nu |t|}
\end{equation}
for any $t\in \mathbb R$.
\end{remark}

\begin{lemma}\label{lGH1} Let $G$ be the Green's
function for hyperbolic semigroup $\{U(t)\}_{t\ge 0}$ and
$\mathcal N$ and $\nu$ be the positive number from (\ref{eqHG1}),
then
\begin{equation}\label{eqHG2}
\int_{0}^{+\infty}\|G(t-\tau)\|d\tau \le \frac{2\mathcal N}{\nu}
\end{equation}
for any $t\in\mathbb R_{+}$, where $\mathcal N$ and $\nu$ there
are two positive number figuring in Definition \ref{defH1} .
\end{lemma}
\begin{proof}
Taking into consideration (\ref{eqHG1}) we obtain
\begin{equation}\label{eqGH3}
\int_{0}^{t}\|G(t-\tau)\|d\tau \le \mathcal N \int_{0}^{t}e^{-\nu
(t-\tau)d\tau}= \frac{\mathcal N}{\nu}(1-e^{-\nu t})
\end{equation}
and
\begin{equation}\label{eqGH4}
\int_{t}^{+\infty}\|G(t-\tau)\|d\tau \le \mathcal N
\int_{t}^{+\infty}e^{\nu (t-\tau)d\tau}= \frac{\mathcal N}{\nu} .
\end{equation}
From (\ref{eqGH3})-(\ref{eqGH4}) we have
\begin{eqnarray}\label{eqGH5}
& \int_{0}^{+\infty}\| G(t-\tau)\|d\tau = \int_{0}^{t}\|
G(t-\tau)\|d\tau +
\int_{t}^{+\infty}\| G(t-\tau)\|d\tau \le \nonumber \\
& \frac{\mathcal N}{\nu}(1-e^{-\nu t}) + \frac{\mathcal N}{\nu}
\le \frac{2\mathcal N}{\nu} \nonumber
\end{eqnarray}
for any $t\in\mathbb R_{+}$. Lemma is proved.
\end{proof}

Denote by $C_{b}(\mathbb R_{+},\mathfrak B)$ the Banach space of
all continuous and bounded mappings $\varphi :\mathbb R_{+}\mapsto
\mathfrak B$ equipped with the norm
$||\varphi||:=\sup\{|\varphi(t)|:\ t\in\mathbb R_{+}\}$.

Let $C_{k}(\mathbb R_{+},\mathfrak B)$ (respectively, $\mathcal L
(\mathbb R_{+},\mathfrak B)$) be the space of all functions
$\varphi \in C(\mathbb R_{+},\mathfrak B)$ such that $\varphi
(\mathbb R_{+})$ is pre-compact in $\mathfrak B$ (respectively,
$\varphi$ is Lagrange stable).

\begin{remark}\label{remLK1}
It easy to check that $C_{k}(\mathbb R_{+},\mathfrak B)$
(respectively, $\mathcal L (\mathbb R_{+},\mathfrak B)$) is a
linear subspace of the Banach space $C(\mathbb R_{+},\mathfrak
B)$.
\end{remark}

\begin{lemma}\label{lLK1} \cite{Sch72} $C_{k}(\mathbb R_{+},\mathfrak B)$ (respectively,
$\mathcal L (\mathbb R_{+},\mathfrak B)$) is a closed subspace of
the space $C(\mathbb R_{+},\mathfrak B)$.
\end{lemma}

\begin{theorem}\label{thLS11_2}
Let $A$ be an infinitesimal generator which generates a
$C_0$-semigroup $\{U(t)\}_{t\ge 0}$ acting on $\mathfrak B$.
Assume that the semigroup $U(t)$ is hyperbolic. Then the following
statements hold:
\begin{enumerate}
\item for every $f\in C_{b}(\mathbb R_{+},\mathbb B)$ there exists
a unique solution $\varphi \in C_{b}(\mathbb R_{+},\mathfrak B)$
of equation \eqref{eqLN1} with $\mathcal P\varphi(0)=0$ given by
the formula
\begin{equation}\label{eqD3}
\varphi(t) =\int_{0}^{+\infty}G(t-\tau)f(\tau)d\tau ;\nonumber
\end{equation}
\item the Green's operator $\mathbb G$ defined by
\begin{equation*}\label{eqG10}
\mathbb G(f)(t):=\int_{0}^{+\infty}G(t-\tau)f(\tau)d\tau
\end{equation*}
is a bounded operator defined on $C_b(\mathbb R_{+},\mathfrak B)$
with values in $C_b(\mathbb R_{+},\mathfrak B)$ and
\begin{equation}\label{Ge}
||\mathbb G(f)||\le \frac{2\mathcal N}{\nu} ||f|| \nonumber
\end{equation}
for any $f\in C_b(\mathbb R_{+}, \mathfrak B)$;

\item if $f\in C_{b}(\mathbb R_{+},\mathfrak B)$ and $l>L>0$, then
\begin{eqnarray}\label{eqL_1}
\max\limits_{0\le t\le L}|\varphi(t)|\le \max\limits_{0\le \tau
\le l}|f(\tau)|\max\limits_{0\le t\le L}\omega_{l}(t)+
\frac{\mathcal N}{\nu}\|f\| e^{\nu (L-l)},
\end{eqnarray}
where
\begin{equation}\label{eqBL02}
\omega_{l}(t):=\frac{\mathcal N}{\nu}(2-e^{-\nu t} -e^{\nu
(t-l)});\nonumber
\end{equation}
\item if $f\in C_{b}(\mathbb R_{+},\mathfrak B)$, then the
solution $\varphi = \mathbb G(f)$ of equation \eqref{eqLN1} is
compatible in the limit.
\end{enumerate}
\end{theorem}
\begin{proof}
(i)-(ii). It is easy to check (see \cite[Ch.II]{Dal} for bounded
operators $\mathcal A$ and \cite{Bas_1997,MRS_1997} for unbounded
operators $\mathcal A$) that the function $\varphi$ given by
\eqref{eqD3} is a solution of equation \eqref{eqLN1}.

Now we will show the boundedness of $\varphi =\mathbb G(f)$. Note
that
\begin{equation*}\label{eqD4}
\varphi(t):= \int _{0}^{+\infty}G(t-\tau)f(\tau)d\tau
\end{equation*}
and we have
\begin{eqnarray}\label{eqB19}
& |\varphi(t)|  =|\int_{0}^{+\infty} G(t-\tau)f(\tau)d\tau | \le
\int_{0}^{+\infty}||G(t-\tau)|||f(\tau)|d\tau \le \nonumber
 \\
& \le  \|f\|\int_{0}^{+\infty} ||G(t-\tau)||d\tau .
\end{eqnarray}
From (\ref{eqB19}) and Lemma \ref{lGH1} it follows
\begin{equation}\label{eqE1}
\|\varphi \|\le 2\frac{\mathcal N}{\nu}\|f\|.\nonumber
\end{equation}

(iii).  Let $L>0$, $t\in [0,L]$, $l>L$ and $f\in C_{b}(\mathbb
R_{+},\mathfrak B)$, then from (\ref{eqB19}) we have
\begin{align}\label{eqBL1}
|\varphi(t)|& \le  \int_{0}^{+\infty}\|G(t-\tau)\| |f(\tau)|d\tau
.
\end{align}
Note that
\begin{eqnarray}\label{eqBL2}
& \int_{0}^{+\infty} \|G(t-\tau)\| |f(\tau)|d\tau =\\
& \int_{0}^{l}\|G(t-\tau)\| |f(\tau)|d\tau +
\int_{l}^{+\infty}\|G(t-\tau)\| |f(\tau)|d\tau ,\nonumber
\end{eqnarray}
\begin{eqnarray}\label{eqBL3}
& \int_{0}^{l}\|G(t-\tau)\| |f(\tau)|d\tau \le \max\limits_{0\le
\tau \le l}|f(\tau)| (\int_{0}^{t}\|G(t-\tau)\| d\tau +
\int_{t}^{l}\|G(t-\tau)\|d\tau) \le \nonumber \\ & \max\limits_{0
\le \tau \le l}|f(\tau)|(\mathcal N \int_{0}^{t}e^{-\nu
(t-\tau)}d\tau + \mathcal N \int_{t}^{l}e^{\nu(t-\tau)}d\tau) =\nonumber \\
&
 \max\limits_{0\le
\tau \le l}|f(\tau)| \frac{\mathcal N}{\nu}(2-e^{-\nu t} -e^{\nu
(t-l)})\nonumber
\end{eqnarray}
and
\begin{align}\label{eqBL4}
\int_{l}^{+\infty}\|G(t-\tau)\| |f(\tau)|d\tau \le \mathcal N
\|f\| \int_{l}^{+\infty}e^{\nu (t-\tau)}d\tau = \frac{\mathcal
N}{\nu}\|f\| e^{\nu (t-l)} .
\end{align}

From (\ref{eqBL1})-(\ref{eqBL4}) we obtain
\begin{eqnarray}\label{eqBL5}
|\varphi(t)|\le \max\limits_{0\le \tau \le l}|f(\tau)|
\frac{\mathcal N}{\nu}(2-e^{-\nu t} -e^{\nu (t-l)}) +
\frac{\mathcal N}{\nu}\|f\| e^{\nu (t-l)}\nonumber
\end{eqnarray}
and, consequently,
\begin{eqnarray}\label{eqBL6}
\max\limits_{0\le t\le L}|\varphi(t)|\le \max\limits_{0\le \tau
\le l}|f(\tau)|\max\limits_{0\le t\le L}\omega_{l}(t)+
\frac{\mathcal N}{\nu}\|f\| e^{\nu (L-l)},\nonumber
\end{eqnarray}
where
\begin{equation}\label{eqBL07}
\omega_{l}(t):=\frac{\mathcal N}{\nu}(2-e^{-\nu t} -e^{\nu (t-l)})
.\nonumber
\end{equation}
Thus inequality (\ref{eqL_1}) is established.

(iv). Let now $\{t_n\}\in \mathfrak L_{f}$, then there exists
$\tilde{f}\in \omega_{f}:=\{g\in C(\mathbb T,\mathfrak B)|\ $
there exists a sequence $\{t_n\}\in \mathfrak L_{f}$ such that
$f^{t_n}\to g$ in $C(\mathbb T,X)$ as $n\to \infty$ $\}$ such that
$f^{t_n}\to \tilde{f}$ in the space $C(\mathbb R_{+},\mathfrak
B)$, that is, for any $L>0$ we have
\begin{eqnarray}\label{eqD6}
& \max\limits_{|t|\le L}|| f(t+t_n)-\tilde{f}(t)||\to 0\nonumber
\end{eqnarray}
as $n\to \infty$.

Denote by $h_{n}(t):=f^{t_n}(t)-\tilde{f}(t)$ for any $t\in\mathbb
R_{+}$, $\varphi_{n}:=\mathbb G(f^{t_n})$,
$\tilde{\varphi}:=\mathbb G(\tilde{f})$ and
$\psi_{n}:=\varphi_{n}-\tilde{\varphi}$. Note that $\psi_n\in
C_{b}(\mathbb R_{+},$ $\mathfrak B)$ is a solution of equation
\begin{equation}\label{eqBL6.0}
x'(t)=Ax(t)+h_{n}(t))\nonumber
\end{equation}
and, consequently, $\psi_{n}=\mathbb G(\tilde{h}_{n})$, $h_{n}\in
C_{b}(\mathbb R_{+},\mathfrak B)$. Since
\begin{eqnarray}\label{eqBL6_01}
 ||\varphi_{n}||=||\mathbb G(f^{t_n})||\le
\frac{2\mathcal N}{\nu}||f^{t_n}||\le \frac{2\mathcal
N}{\nu}||f||:=C\nonumber
\end{eqnarray}
and
\begin{equation}\label{eqBL6_02}
\|h_{n}\|=\|f^{t_{n}}-\tilde{f}\|\le \|f^{t_{n}}\| +
\|\tilde{f}\|\le 2\|f\| \nonumber
\end{equation}
for any $n\in\mathbb N$.

Let now $\{l_n\}$ be a sequence of positive numbers such that
$l_n\to \infty$ as $n\to \infty$. According to inequality
(\ref{eqL_1}) we obtain
\begin{eqnarray}\label{eqBL6_03}
&\max\limits_{0\le t\le L}|\varphi_{n}(t)|\le \max\limits_{0\le
\tau \le l_{n}}|h_{n}(\tau)|\max\limits_{0\le t\le
L}\omega_{l_{n}}(t)+ \frac{\mathcal N}{\nu}\|f\| e^{\nu
(L-l_{n})}\le \nonumber \\
& \frac{2\mathcal N}{\nu}\max\limits_{0\le \tau \le
l_{n}}|h_{n}(\tau)| + \frac{\mathcal N}{\nu}\|f\| e^{\nu
(L-l_{n})}
\end{eqnarray}
because
\begin{equation}\label{eqBL7}
\omega_{l_{n}}(t)=\frac{\mathcal N}{\nu}(2-e^{-\nu t} -e^{\nu
(t-l_{n})})\le \frac{2\mathcal N}{\nu}
\end{equation}
for any $t\in [0,L]$ and $n\in \mathbb N$.

Passing to the limit in (\ref{eqBL6_03}) and taking into
consideration that $\max\limits_{0\le \tau \le
l_{n}}|h_{n}(\tau)|\to 0$ as $n\to \infty$ (see Remark
\ref{remD1}, item 2) and (\ref{eqBL7}) we obtain
\begin{equation}\label{eqBL8}
\lim\limits_{n\to \infty}\max\limits_{0\le t\le
L}|\varphi_{n}(t)|=0 \nonumber
\end{equation}
for any $L>0$, i.e., $\varphi^{t_n}\to \tilde{\varphi}$ in the
space $C(\mathbb R_{+},\mathfrak B)$ as $n\to \infty$. This means
that $\mathfrak L_{f}\subseteq \mathfrak L_{\varphi}$. The proof
is complete.
\end{proof}

\begin{coro}\label{cor 10}
Under the conditions of \textsl{Theorem} \ref{thLS11_2} if the
function $f\in$ $ C_{b}(\mathbb R_{+},$ $\mathfrak B)$ is
asymptotically stationary (respectively, asymptotically
$\tau$-periodic, asymptotically quasi-periodic with the spectrum
of frequencies $\nu_1,\ldots,\nu_k$, asymptotically Bohr almost
periodic, asymptotically Bohr almost automorphic, asymptotically
Birkhoff recurrent, Lagrange stable), then equation (\ref{eqLN1})
has a unique solution $\varphi \in C_{b}(\mathbb R_{+},\mathfrak
B)$ with $\mathcal P \varphi(0)=0$ which is asymptotically
stationary (respectively, asymptotically $\tau$-periodic,
asymptotically quasi-periodic with the spectrum of frequencies
$\nu_1,\ldots,\nu_k$, asymptotically Bohr almost periodic,
asymptotically Bohr almost automorphic, asymptotically Birkhoff
recurrent, Lagrange stable).
\end{coro}

\begin{proof}
This statement follows from Theorems \ref{th2} and \ref{thLS11_2}.
\end{proof}

\begin{remark}\label{remLP1}  1. Let $\varphi \in C(\mathbb R_{+},\mathfrak
B)$ and we denote by $\phi_{\varphi}\in C(\mathbb R,\mathfrak B)$
a function defined as follows: $\phi_{\varphi}(t)=\varphi(t)$ for
any $t\in \mathbb R_{+}$ and $\phi_{\varphi}(t)=\varphi(0)$ for
any $t\in \mathbb R_{-}:=(-\infty,0]$.

2. The space $C(\mathbb R_{+},\mathfrak B)$ (respectively,
$C_{b}(\mathbb R_{+},\mathfrak B)$, $C_{k}(\mathbb R_{+},\mathfrak
B)$ or $\mathfrak L(\mathbb R_{+},\mathfrak B)$) can be embedded
continuously in the space $C(\mathbb R,\mathfrak B)$
(respectively, $C_{b}(\mathbb R,\mathfrak B)$, $C_{k}(\mathbb
R,\mathfrak B)$ or $\mathfrak L(\mathbb R,\mathfrak B)$) as
follows: $\varphi \to \phi_{\varphi}$.
\end{remark}

\begin{lemma}\label{lLP01}
If $\varphi \in C(\mathbb R_{+},\mathfrak B)$, $t_n\to +\infty$
and $\varphi^{t_n}\to \tilde{\varphi}$ in $C(\mathbb
R_{+},\mathfrak B)$ as $n\to \infty$, then the following
statements hold:
\begin{enumerate}
\item $\tilde{\varphi}\in C(\mathbb R,\mathfrak B)$; \item if
$\varphi \in C_{b}(\mathbb R_{+},\mathfrak B)$, then
$\tilde{\varphi}\in C_{b}(\mathbb R,\mathfrak B)$ and
$|\tilde{\varphi}(t)|\le \|\varphi\|:=\sup\limits_{t\in \mathbb
R_{+}}|\varphi(t)|$ for any $t\in \mathbb R$;\item if
$Q:=\overline{\varphi(\mathbb R_{+})}$ is a compact subset of
$\mathfrak B$, then $\tilde{\varphi}(\mathbb R)\subseteq Q$; \item
if $\varphi \in \mathcal L(\mathbb R_{+},\mathfrak B)$, then
$\Sigma_{\varphi}^{+}:=\{\sigma(\tau,\varphi)|\ \tau \ge 0\}$ is a
pre-compact subset of $C(\mathbb R,\mathfrak B)$.
\end{enumerate}
\end{lemma}
\begin{proof} The first three statements are evident. To prove the
last statement without loss of generality we can suppose that
$\varphi^{t_n}\in C(\mathbb R,\mathfrak B)$
($\varphi^{t_n}(t):=\varphi(t_n)$ for any $t\le -t_{n}$). Since
the function $\varphi \in \mathcal L(\mathbb R_{+},\mathfrak B)$,
then by Lemma \ref{lLS1} it is uniformly continuous on $\mathbb
R_{+}$ and $Q=\overline{\varphi(\mathbb R_{+})}$ is a compact
subset of $\mathfrak B$. Let $\{t_n\}\subset \mathbb R_{+}$. Since
$\varphi \in \mathcal L (\mathbb R_{+},\mathfrak B)$, then without
loss of generality we can suppose that the sequence $\{t_n\}$
converges to $+\infty$ as $n\to \infty$ and $\{\varphi^{t_n}\}$
converges in $C(\mathbb R_{+},\mathfrak B)$. If we consider
$\{\varphi^{t_n}\}$ as a sequence from $C(\mathbb R,\mathfrak B)$,
then it is possible to check that the sequence $\{\varphi^{t_n}\}$
is equi-continuous on every segment $[-l,l]$ ($l\in \mathbb R$)
and $\varphi^{t_n}(\mathbb R)\subseteq Q=\overline{\varphi(\mathbb
R_{+})}$ and, consequently, the sequence $\{\varphi^{t_n}\}$ is
pre-compact in $C(\mathbb R,\mathfrak B)$. Since the sequence
$\{\varphi^{t_n}\}$ converges in $C(\mathbb R_{+},\mathfrak B)$,
then it is easy to see that the sequence $\{\varphi^{t_n}\}$ has
at most one (in fact exactly one) limiting point. This means that
the sequence $\{\varphi^{t_n}\}$ converges in the space $C(\mathbb
R,\mathfrak B)$ and, consequently, $\Sigma_{\varphi}^{+}$ is a
pre-compact subset of $C(\mathbb R,\mathfrak B)$.

\end{proof}

\begin{theorem}\label{thLS11_3}
Let $A$ be an infinitesimal generator which generates a
$C_0$-semigroup $\{U(t)\}_{t\ge 0}$ acting on $\mathfrak B$.
Assume that the semigroup $U(t)$ is hyperbolic. Then the following
statements hold:
\begin{enumerate}
\item for every $f\in C_{b}(\mathbb R,\mathbb B)$ there exists a
unique solution $\varphi \in C_{b}(\mathbb R_{+},\mathfrak B)$ of
equation \eqref{eqLN1} given by the formula
\begin{equation}\label{eqD03}
\varphi(t) =\int_{-\infty}^{+\infty}G(t-\tau)f(\tau)d\tau
;\nonumber
\end{equation}
\item the Green's operator $\mathbb G$ defined by
\begin{equation*}\label{eqG_10}
\mathbb G(f)(t):=\int_{-\infty}^{+\infty}G(t-\tau)f(\tau)d\tau
\end{equation*}
is a bounded operator defined on $C_b(\mathbb R,\mathfrak B)$ with
values in $C_b(\mathbb R,\mathfrak B)$ and
\begin{equation}\label{eqG_11}
||\mathbb G(f)||_{\infty}\le \frac{2\mathcal N}{\nu} ||f||
\nonumber
\end{equation}
for any $f\in C_b(\mathbb R, \mathfrak B)$;

\item if $f\in C_{b}(\mathbb R,\mathfrak B)$ and $l>L>0$, then
\begin{equation}\label{eqL_01}
 \max\limits_{|t| \le L}|\varphi(t)|\le \frac{\mathcal
N}{\nu}\big{(} \max\limits_{ |\tau| \le l}|f(\tau)|\max\limits_{
|t| \le L}(2-\omega_{l}(t))+\|f\|\max\limits_{ |t| \le
L}\omega_{l}(t)\big{)} ,\nonumber
\end{equation}
where
\begin{equation}\label{eqBL_02}
 \omega_{l}(t):=e^{-\nu (t-l)}+e^{\nu ((t+l)}\nonumber .
\end{equation}
\end{enumerate}
\end{theorem}
\begin{proof} This statement can be proved analogically as Theorem \ref{thLS11_2} (with slight
modifications).
\end{proof}

\begin{lemma}\label{lLP1}
Let $A$ be an infinitesimal generator which generates a
$C_0$-semigroup $\{U(t)\}_{t\ge 0}$ acting on $\mathfrak B$.
Assume that
\begin{enumerate}
\item the semigroup $U(t)$ is hyperbolic; \item $\{f_n\}\subset
C_{b}(\mathbb R_{+},\mathfrak B)$; \item there exists $C\in
\mathbb R_{+}$ such that $\|f_n\|\le C$ for any $n\in \mathbb N$;
\item $\{f_n\}$ converges to $f$ in $C(\mathbb R_{+},\mathfrak B)$
as $n\to \infty$; \item $f\in C_{b}(\mathbb R_{+},\mathfrak B)$
and $\{t_n\}$ such that $t_n\to +\infty$ and $f^{t_n}\to
\tilde{f}$ in $C(\mathbb R_{+},\mathfrak B)$ as $n\to \infty$.
\end{enumerate}

Then
\begin{enumerate}
\item the sequence $\{\varphi_n\}$ converges to $\varphi$ in
$C(\mathbb R_{+},\mathfrak B)$ as $n\to \infty$, where
\begin{equation}\label{eqLP1}
\varphi_n(t) =\int_{0}^{+\infty}G(t-\tau)f_n(\tau)d\tau
\end{equation}
and
\begin{equation}\label{eqLP2}
\varphi(t) =\int_{0}^{+\infty}G(t-\tau)f(\tau)d\tau
\end{equation}
for any $t\in \mathbb R_{+}$ and $n\in \mathbb N$; \item
$\varphi^{t_n}\to \tilde{\varphi}$ in $C(\mathbb R,\mathfrak B)$
as $n\to \infty$, where
\begin{equation}\label{eqLP2.1}
\tilde{\varphi}(t)
=\int_{-\infty}^{+\infty}G(t-\tau)\tilde{f}(\tau)d\tau
\end{equation}
for any $t\in \mathbb R$.
\end{enumerate}
\end{lemma}
\begin{proof} Let $\{l_n\}$ be an arbitrary sequence such that $l_n\to
+\infty$ as $n\to \infty$ and $L$ be an arbitrary fixed positive
number. Denote by $n_0\in \mathbb N$ such that $t_n >L$ for any
$n\ge n_0$. Taking into consideration (\ref{eqLP1})-(\ref{eqLP2})
and Theorem \ref{thLS11_2} (item (iii)) we will have
\begin{equation}\label{eqLP3}
\max\limits_{0\le t\le L}|\varphi_{n}(t)-\varphi(t)|\le
\frac{2\mathcal N}{\nu}\max\limits_{0\le \tau \le
l_n}|f_n(\tau)-f(t)|+\frac{2C\mathcal N}{\nu}e^{\nu (L-l_n)}
\end{equation}
for any $n\ge n_0$. Since $f_n\to f$ as $n\to \infty$, then by
Remark \ref{remD1} (item 2.) we have
\begin{equation}\label{eqLP4}
\lim\limits_{n\to \infty}\max\limits_{0\le \tau \le
l_n}|f_n(\tau)-f(t)|=0 .
\end{equation}
Passing to the limit in (\ref{eqLP3}) as $n\to \infty$ and taking
into consideration (\ref{eqLP4}) we obtain
\begin{equation}\label{eqLP5}
\max\limits_{0\le t\le L}|\varphi_{n}(t)-\varphi(t)|=0
\end{equation}
and, consequently, $\varphi_n\to \varphi$ as $n\to \infty$ in
$C(\mathbb R_{+},\mathfrak B)$.

Note that the function $\tilde{f}\in C(\mathbb R,\mathfrak B)$ and
$|\tilde{f}(t)|\le \sup\{|f(t)|:\ t\in \mathbb R_{+}\}$ for any
$t\in \mathbb R$, that is, $\tilde{f}\in C_{b}(\mathbb R,\mathfrak
B)$. By Theorem \ref{thLS11_3} equation
\begin{equation}\label{eqLP6}
x'=\mathcal A x +\tilde{f}(t)
\end{equation}
has a unique solution $\tilde{\varphi}\in C_{b}(\mathbb
R,\mathfrak B)$ defined by equality (\ref{eqLP2.1}). From
(\ref{eqLP2}) we have
\begin{equation}\label{eqLP7}
\varphi(t+t_n) =\int_{0}^{+\infty}G(t+t_n-\tau)f(\tau)d\tau
=\int_{-t_n}^{+\infty}G(t-\tau)f(\tau +t_n)d\tau
\end{equation}
for any $t\in [-t_n,+\infty)$. Let $L$ be an arbitrary positive
fixed number and $n_0\in\mathbb N$ be chosen such that $t_n\ge L$
for any $n\ge n_0$. According to (\ref{eqLP2.1}) and (\ref{eqLP7})
we get
\begin{eqnarray}\label{eqLP08}
& I_{n}(t):=|\varphi(t+t_n)-\tilde{\varphi}(t)|=\nonumber
\\
&|\int_{-t_n}^{+\infty}G(t-\tau)[f(\tau
+t_n)-\tilde{f}(\tau)]d\tau - \int_{-\infty}^{-t_n}G(t-\tau)\tilde{f}(\tau) |\le \nonumber \\
& |\int_{-t_n}^{+\infty}G(t-\tau)[f(\tau +t_n)-\tilde{f}(\tau)]d\tau |+\nonumber \\
& |\int_{-\infty}^{-t_n}G(t-\tau)[f(\tau
+t_n)-\tilde{f}(\tau)]d\tau |= I^{1}_{n}(t)+I^{2}_{n}(t),\nonumber
\end{eqnarray}
where
\begin{equation}\label{eqLP09}
I_{n}^{1}(t):=|\int_{-t_n}^{+\infty}G(t-\tau)[f(\tau
+t_n)-\tilde{f}(\tau)]d\tau |\nonumber
\end{equation}
and
\begin{equation}\label{eqLP801}
I_{n}^{2}(t):=|\int_{-\infty}^{-t_n}G(t-\tau)\tilde{f}(\tau)d\tau
|\nonumber
\end{equation}
for any $t\in [-t_n,+\infty)$.

We can estimate the integral $I_{n}^{2}(t)$ as follows
\begin{eqnarray}\label{eqLP8}
& I_{n}^{2}(t):=|\int_{-\infty}^{-t_n}G(t-\tau)\tilde{f}(\tau)d\tau |\le \\
& 2\mathcal N \|f\|\int_{-\infty}^{-t_n}e^{-\nu(t-\tau)}d\tau
=\frac{2\mathcal N\|f\|}{\nu}e^{-\nu (t+t_n)} \nonumber
\end{eqnarray}
for any $t\in \mathbb R$ and $n\in \mathbb N$, where
$\|f\|:=\sup\limits_{t\in\mathbb R_{+}}|f(t)|$. From (\ref{eqLP8})
we obtain
\begin{eqnarray}\label{eqLP9}
\max\limits_{|t|\le L}I_{n}^{2}(t)\le \frac{2\mathcal
N\|f\|}{\nu}e^{-\nu t_n}e^{\nu L}
\end{eqnarray}
for any $n\in \mathbb N$. Passing to the limit in (\ref{eqLP9}) as
$n\to \infty$ we get
\begin{eqnarray}\label{eqLP10}
\lim\limits_{n\to \infty}\max\limits_{|t|\le
L}I_{n}^{2}(t)=0\nonumber
\end{eqnarray}
for any $L>0$.

To estimate the integral $I_n^{1}(t)$ we note that
\begin{eqnarray}\label{eqLP11}
&
I_{n}^{1}(t):=|\int_{-t_n}^{+\infty}G(t-\tau)\hat{f}_{n}(\tau)d\tau
|\le \nonumber \\ &
|\int_{-t_n}^{-l}G(t-\tau)\hat{f}_{n}(\tau)d\tau | +\nonumber
\\
& |\int_{-l}^{l}G(t-\tau)\hat{f}_{n}(\tau)d\tau | +
 |\int_{l}^{+\infty}G(t-\tau)\hat{f}_{n}(\tau)d\tau | = \nonumber \\
& I_{n}^{11}(t)+I_{n}^{12}(t) + I_{n}^{13}(t),\nonumber
\end{eqnarray}
where $\hat{f}_{n}(t):=f(t+t_n)-\tilde{f}(t)$ for any $t\in
\mathbb R$ and $\|\hat{f}_{n}\|\le 2\|f\|$ for any $n\in \mathbb
N$,
\begin{equation}\label{eqLP12}
I^{11}_{n}(t):=|\int_{-t_n}^{-l}G(t-\tau)\hat{f}_{n}(\tau)d\tau
|,\nonumber
\end{equation}
\begin{equation}\label{eqLP13}
I^{12}_{n}(t):=|\int_{-l}^{l}G(t-\tau)\hat{f}_{n}(\tau)d\tau
|\nonumber
\end{equation}
and
\begin{equation}\label{eqLP14}
I^{13}_{n}(t):=|\int_{l}^{+\infty}G(t-\tau)\hat{f}_{n}(\tau)d\tau
| .\nonumber
\end{equation}
Since
\begin{eqnarray}\label{eqLP15}
& I_{n}^{11}(t)\le \mathcal N \int_{-t_n}^{-l}e^{-\nu
(t-\tau)}|\hat{f}_{n}(\tau)|d\tau \le 4\mathcal N \|f\|e^{-\nu
t}\frac{e^{\nu \tau}}{\nu}\Big{|}_{-t_n}^{-l}\le \nonumber
\\
& \frac{4\mathcal N \|f\|}{\nu}e^{-\nu t}(e^{-\nu l}-e^{-\nu
t_n})\nonumber
\end{eqnarray}
and, consequently,
\begin{equation}\label{eqLP16}
\limsup\limits_{n\to \infty}\max\limits_{|t|\le L}I^{11}_{n}(t)\le
\frac{4\mathcal N \|f\|}{\nu}e^{\nu (L-l)}
\end{equation}
for any $l>0$. Taking in consideration that $l$ is an arbitrary
positive number from (\ref{eqLP16}) we obtain
\begin{equation}\label{eqLP17}
\lim\limits_{n\to \infty}\max\limits_{|t|\le L}I^{11}_{n}(t)=0 .
\end{equation}
For $I^{12}_{n}(t)$ we have
\begin{equation}\label{eqLP018}
I^{12}_{n}(t):=|\int_{-l}^{l}G(t-\tau)\hat{f}_{n}(\tau)d\tau |\le
\mathcal N\max\limits_{|\tau|\le l}|\hat{f}_{n}(\tau)|
\max\limits_{|\tau|\le l}\int_{-l}^{l}e^{-\nu |t-\tau|}d\tau
,\nonumber
\end{equation}
hence, we get
\begin{eqnarray}\label{eqLP18}
& \max\limits_{|t|\le L}I^{12}_{n}(t)=
|\int_{-l}^{l}G(t-\tau)\hat{f}_{n}(\tau)d\tau | \le \nonumber \\
& \mathcal N \max\limits_{|\tau|\le
l}|\hat{f}_{n}(\tau)|\max\limits_{|t|\le L}\int_{-l}^{l}e^{-\nu
|t-\tau|}d\tau \to 0
\end{eqnarray}
as $n\to \infty$.

Finally, for the integral $I^{13}_{n}$ we obtain
\begin{equation}\label{eqLP19}
I^{13}_{n}(t):=|\int_{l}^{+\infty}G(t-\tau)\hat{f}_{n}(\tau)d\tau
| \le \frac{4\mathcal N \|f\|}{\nu}e^{\nu (t-l)}\nonumber
\end{equation}
and, consequently,
\begin{equation}\label{eqLP20}
\limsup\limits_{n\to \infty}\max\limits_{|t|\le L}I_{n}^{13}(t)\le
\frac{4\mathcal N \|f\|}{\nu}e^{\nu (L-l)}
\end{equation}
for any $l>0$. From (\ref{eqLP20}) we get
\begin{equation}\label{eqLP21}
\lim\limits_{n\to \infty}\max\limits_{|t|\le L}I_{n}^{13}(t) = 0 .
 \end{equation}
From (\ref{eqLP17}), (\ref{eqLP18}) and (\ref{eqLP21}) we obtain
\begin{equation}\label{eqLP22}
\lim\limits_{n\to \infty}\max\limits_{|t|\le
L}I^{1}_{n}(t)=0\nonumber
\end{equation}
for any $L>0$ and, hence,
\begin{equation}\label{eqLP23}
\lim\limits_{n\to \infty}\max\limits_{|t|\le L}I_{n}(t)=0
.\nonumber
\end{equation}
Lemma is completely proved.
\end{proof}

\section{Bounded solutions on the Semi-Axis for Semi--Linear Differential
Equations}\label{S4}

In this section we establish the conditions, under which the
existence of a compatible in the limit solution of a semi-linear
equation with hyperbolic linear part.

Denote by $\Delta$ the family of all increasing (nondecreasing)
functions $L :\mathbb R_{+}\to\mathbb R_{+}$ with
$\lim\limits_{r\to 0^{+}}L(r)=0$.

\begin{definition}\label{defLE1} A function $F\in C(\mathbb T\times\mathfrak B,\mathfrak
B)$ is said to be
\begin{enumerate}
\item local Lipschitzian with respect to variable $u\in \mathfrak
B$ on $\mathfrak B$ uniformly with respect to $t\in \mathbb T$, if
there exists a function $L:\mathbb R_{+}\to \mathbb R_{+}$ such
that
\begin{equation}\label{eqLE1}
|F(t,u_1)-F(t,u_2)|\le L(r)|u_1-u_2|
\end{equation}
for all $u_1,u_2\in B[0,r]$ and $t\in\mathbb T$, where
$B[0,r]:=\{u\in \mathfrak B:\ |u|\le r\}$; \item (global)
Lipschitzian with respect to variable $u\in \mathfrak B$ on
$\mathfrak B$ uniformly with respect to $t\in \mathbb T$, if there
exists a positive constant $L$ such that
\begin{equation}\label{eqLE1.1}
|F(t,u_1)-F(t,u_2)|\le L|u_1-u_2|
\end{equation}
for all $u_1,u_2\in \mathfrak B$ and $t\in\mathbb T$.
\end{enumerate}
\end{definition}

\begin{definition}\label{defLE2} The smallest constant figuring in
(\ref{eqLE1}) (respectively, in (\ref{eqLE1.1})) is called
Lipschitz constant of function $F$ on $\mathbb T\times B[0,r]$
(respectively, on $\mathbb T\times \mathfrak B$). Notation
$Lip(r,F)$ (respectively, $Lip(F)$).
\end{definition}

\begin{remark}\label{remL1} Note the function $Lip(\cdot,F):\mathbb R_{+}\to \mathbb
R_{+}$ is nondecreasing, that is, if for all $r_1,r_2\in \mathbb
R_{+}$ such that $r_1\le r_2$ one has $Lip(r_1,F)\le Lip(r_2,F)$.
\end{remark}

\begin{lemma}\label{lLC1} Assume that
$F\in C(\mathbb T\times \mathfrak B,\mathfrak B)$ and it is local
(respectively, global) Lipschitzian with respect to variable $u\in
\mathfrak B$ on $\mathfrak B$ uniformly with respect to $t\in
\mathbb T$, i.e., there exists a function $L:\mathbb R_{+}\to
\mathbb R_{+}$ (respectively, a positive constant $L$) such that
(\ref{eqLE1}) holds (respectively, (\ref{eqLE1.1})) for all
$u_1,u_2\in \mathfrak B$ and $t\in\mathbb T$.

Then every function $G\in H(F);=\overline{\{\sigma(\tau,F)|\ \tau
\in \mathbb T\}}$ is local (respectively, global) Lipschitzian
with respect to variable $u\in \mathfrak B$ on $\mathfrak B$
uniformly with respect to $t\in \mathbb T$ with the same function
$L(\cdot,F):\mathbb R_{+}\to \mathbb R_{+}$ (respectively, the
same positive constant $L$).
\end{lemma}
\begin{proof} This statement is evident.
\end{proof}

Consider an evolutionary differential equation
\begin{equation}\label{eqSL1}
u'=Au + F(t,u)
\end{equation}
in the Banach space $\mathfrak B$, where $F$ is a nonlinear
continuous mapping ("small" perturbation) acting from $\mathbb
R_{+}\times \mathfrak B$ into $\mathfrak B$.

A function $u:[0,a)\mapsto \mathfrak B$ is said to be a weak
(mild) solution of equation (\ref{eqSL1}) passing through the
point $x\in \mathfrak B$ at the initial moment $t=0$ (notation
$\varphi(t,x,F)$) if $u\in C([0,T],\mathfrak B)$ and satisfies the
integral equation
\begin{equation}\label{eqSL3}
u(t)=U(t)x+\int_{0}^{t}U(t-s)F(s,u(s))ds \nonumber
\end{equation}
for all $t\in [0,T]$ and $0<T<a$.

\begin{definition}\label{defSL01}
A function $F\in C(\mathbb R_{+}\times \mathfrak B,\mathfrak B)$
is said to be regular, if for any $x\in \mathfrak B$ there exists
a unique solution $\varphi(t,x,F)$ of equation (\ref{eqSL1})
passing through the point $x$ at the initial moment $t=0$ and
defined on $\mathbb R_{+}$.
\end{definition}

Below everywhere in this Section we suppose that the function
$F\in C(\mathbb R_{+}\times \mathfrak B,\mathfrak B)$ is regular.

\begin{remark}\label{remRF1} If the function $F\in C(\mathbb R_{+}\times \mathfrak B,\mathfrak
B)$ is global Lipschitzian and $\sup\limits_{t\ge
0}|F(t,0)|<+\infty$, then it is regular (see Theorem 6.6.9 from
\cite[Ch.VI]{Che_2020}).
\end{remark}

Let $\mathfrak{L}$ some set of sequences $\{t_k\}\to+\infty$ and
$r>0$. Denote by
\begin{enumerate}
\item[-] $C_{r}(\mathfrak{L}):= \{\varphi : \varphi\in
C_{b}(\mathbb{R}_+,\mathfrak B),\ \mathfrak{L}\subseteq
\mathfrak{L}_{\varphi}\ \mbox{and}\ ||\varphi||\leq r \}$;
\item[-] $BC(\mathfrak L):=\{\varphi \in C_{b}(\mathbb
R_{+},\mathfrak B):\ \mathfrak L\subseteq \mathfrak
L_{\varphi}\}$; \item[-] $\mathcal L (\mathfrak L):=\{\varphi \in
\mathcal L(\mathbb R_{+},\mathfrak B):\ \mathfrak L \subseteq
\mathfrak L_{\varphi}\}$; \item[-] $\mathcal L_{r}(\mathfrak
L):=\{\varphi \in \mathcal L (\mathfrak L):\ \mbox{and}\
||\varphi||\leq r \}$.
\end{enumerate}

\begin{lemma}\label{l3.4.1} \cite[Ch.III]{Che_2009}, \cite{Sch72}
$C_{r}(\mathfrak{L})$ (respectively, $C_{b}(\mathfrak L)$ or
$\mathcal L(\mathfrak L)$) is a closed subspace of the metric
space $C_b(\mathbb{R}_+,\mathfrak B)$.
\end{lemma}

\begin{lemma}\label{l3.4.2} Let $\varphi\in \mathcal L(\mathbb{R}_+,\mathfrak B)$
(respectively, $\varphi \in C_{b}(\mathbb R_{+},\mathfrak B)$),
$F\in \mathcal L(\mathbb{R_{+}}\times \mathfrak B,\mathfrak B)$
(respectively, $F\in C_{b}(\mathbb R_{+}\times \mathfrak
B,\mathfrak B)$), $g(t):=F(t,\varphi(t))$ for all
$t\in\mathbb{R}_+$ and $F_Q:=F|_{\mathbb{R}\times Q}$, where
$Q:=\overline{\varphi(\mathbb{R}_{+})}$. If $F_Q$ satisfies the
condition of Lipschitz with respect to the second variable with
the constant $L>0$, then
\begin{enumerate}
\item $\mathfrak L_{(F_{Q},\varphi)}\subseteq \mathfrak L_{g}$;
\item if $\{t_n\}\in \mathfrak L_{(F_{Q},\varphi)}$ and
$$
(\tilde{\varphi},\tilde{F})=\lim\limits_{n\to
\infty}(\varphi^{t_n},F_{Q}^{t_n}),
$$
then $\lim\limits_{n\to \infty}g^{t_n}=\tilde{g}$, where
$\tilde{g}(t):=\tilde{F}(t,\tilde{\varphi}(t))$ for any $t\in
\mathbb R$; \item $g\in \mathcal L(\mathbb R_{+},\mathfrak B)$
(respectively, $g\in C_{b}(\mathbb R_{+},\mathfrak B)$).
\end{enumerate}
\end{lemma}
\begin{proof} Let $\{t_n\}\in \mathfrak L_{(F_{Q},\varphi))}$, then
we have
\begin{eqnarray}\label{eqLF1}
& |g(t+t_n)-\tilde{g}(t)|\le
|F(t+t_n,\varphi(t+t_n))-\tilde{F}(t,\varphi(t+t_n))|+ \nonumber
\\
&|\tilde{F}(t,\varphi(t+t_n))-\tilde{F}(t,\tilde{\varphi}(t))|\le
\\
& \max\limits_{x\in
Q}|F(t+t_n,x)-\tilde{F}(t,x)|+L|\varphi(t+t_n)-\tilde{\varphi}(t)|\nonumber
\end{eqnarray}
for any $n\in \mathbb N$ and $t\in\mathbb R_{+}$. From
(\ref{eqLF1}) we get
\begin{eqnarray}\label{eqLF2}
& \max\limits_{0\le t\le l}|g(t+t_n)-\tilde{g}(t)|\le
\max\limits_{x\in Q}\max\limits_{0\le t\le
l}|F(t+t_n,x)-\tilde{F}(t,x)|+\nonumber \\
& L\max\limits_{0\le t\le l}|\varphi(t+t_n)-\tilde{\varphi}(t)|
\end{eqnarray}
for any $l>0$. Passing to the limit in (\ref{eqLF2}) we conclude
that the sequence $\{g^{t_n}\}$ converges $\tilde{g}$ in
$C(\mathbb R_{+},\mathfrak B)$, i.e., $\{t_n\}\in \mathfrak
L_{g}$.

Let $\{t_n\}$ ($t_n\in \mathbb R_{+}$) be an arbitrary sequence.
We will show that the sequence $\{g^{t_n}\}$ is pre-compact in
$\mathcal L(\mathbb R_{+},\mathfrak B)$. Without loss of
generality we can suppose that the sequence $\{t_n\}$ goes to
$+\infty$ as $n\to \infty$. Since the functions $\varphi$ and $F$
are Lagrange stable, then without loss of generality we may assume
that the sequence $\{t_n\}\in \mathfrak L_{(F_{Q},\varphi)}$.
According to the first statement of Lemma \ref{l3.4.2} we have
$\{t_n\} \in \mathfrak L_{g}$ and, consequently, $g\in \mathcal
L(\mathbb R_{+},\mathfrak B)$.

To finish the proof of Lemma it is sufficient to show that if the
functions $f$ and $F$ are bounded, then $g$ is so. In fact. Since
\begin{equation}\label{eqBd1}
|g(t)|=|F(t,0)+F(t,\varphi(t))-F(t,0)|\le |F(t,0)|+L|\varphi(t)|
\nonumber
\end{equation}
for any $t\in \mathbb R_{+}$ and, consequently,
\begin{equation}\label{eqBd2}
\sup\limits_{t\in\mathbb R_{+}}|g(t)|\le \sup\limits_{t\in\mathbb
R_{+}}|F(t,0)|+L\sup\limits_{t\in\mathbb R_{+}}|\varphi(t)|
.\nonumber
\end{equation}
This means that $g\in C_{b}(\mathbb R_{+},\mathfrak B)$. Lemma is
proved.
\end{proof}

Let us consider a differential equation
\begin{equation}\label{eqSL01}
u'=Au + f(t) +F(t,u),
\end{equation}
where $F\in C(\mathbb{R_{+}}\times \mathfrak B,\mathfrak B)$ and
$F(t,0)=0$ for any $t\in \mathbb R_{+}$.

We put $\mathfrak B_{-}:=\mathcal P(\mathfrak B)$, where $\mathcal
P$ is a projection figuring in the Definition \ref{defH1}.

Let $\varphi_{0}$ be a unique compatible in the limit solution of
equation
\begin{equation}\label{eqSL1.1}
u'=Au + f(t)
\end{equation}
with $\mathcal P\varphi_{0}(0)=0$ the existence of which is
guarantied by Theorem \ref{thLS11_2}. Assume
$Q:=\overline{\varphi_{0}(\mathbb{R}_+)}$ and by $Q_{r}:=\{x\in
\mathfrak B|\ \rho(x,Q)\le r\}$, where $\rho(x,Q):=\inf\{|x-q|:\
q\in Q\}$, denote a neighborhood of the set $Q\subset \mathfrak B$
of radius $r>0$.

\begin{definition}\label{defB1.0} A function $F\in C(\mathbb R_{+}\times \mathfrak B,\mathfrak
B)$ is said to be bounded (respectively, compact) if for any
bounded subset $B$ (respectively, compact subset $Q$) from
$\mathfrak B$ the set $F(\mathbb R_{+}\times B):=\{F(t,x)|\
(t,x)\in \mathbb R_{+}\times B\}$ (respectively, $F(\mathbb
R_{+}\times Q)$) is bounded (respectively, pre-compact) in
$\mathfrak B$.
\end{definition}

Let $\varphi\in C(\mathbb R_{+},\mathfrak B)$ be a compact
solution of equation (\ref{eqSL1.1}).

\begin{definition}\label{defSL1} The solution $\varphi$ of equation (\ref{eqSL1.1}) is said to
be compatible by the character of recurrence in the limit (shortly
compatible in the limit) if $\mathfrak L_{(f,F_{Q})}\subseteq
\mathfrak L_{\varphi}$, where $Q:=\overline{\varphi(\mathbb
R_{+})}$ and $F_{Q}:=F\big{|}_{\mathbb R_{+}\times Q}$.
\end{definition}

\begin{theorem}\label{thSL0.1} Assume that the following conditions are
fulfilled:
\begin{enumerate}
\item $\varphi$ is a compact solution of equation (\ref{eqSL1.1});
\item the functions $f$ and $F_{Q}$ are jointly asymptotically
stationary (respectively, asymptotically $\tau$-periodic,
asymptotically quasi-periodic with the spectrum of frequencies
$\nu_1,\ldots,\nu_k$, asymptotically Bohr almost periodic,
asymptotically Bohr almost automorphic, asymptotically Birkhoff
recurrent); \item $\varphi$ is a compatible in the limit solution
of equation (\ref{eqSL1.1}).
\end{enumerate}

Then the solution $\varphi$ is also asymptotically stationary
(respectively, asymptotically $\tau$-periodic, asymptotically
quasi-periodic with the spectrum of frequencies
$\nu_1,\ldots,\nu_k$, asymptotically Bohr almost periodic,
asymptotically Bohr almost automorphic, asymptotically Birkhoff
recurrent).
\end{theorem}
\begin{proof} This statement directly follows from the Definition
\ref{defSL1} and Theorem \ref{th2}.
\end{proof}

\begin{theorem}\label{th4A}
Let $f\in C(\mathbb{R}_+,\mathfrak B)$, $F\in C(\mathbb{R}_+\times
\mathfrak B,\mathfrak B)$ and $A$ be an infinitesimal generator
which generates a $C_0$-semigroup $\{U(t)\}_{t\ge 0}$ acting on
$\mathfrak B$.

Assume that the following conditions are fulfilled:
\begin{enumerate}
\item[$1)$] the semigroup $\{U(t)\}_{t\ge 0}$ is hyperbolic;
\item[$2)$] the functions $f$ and $F$ are Lagrange stable;
\item[$3)$] $F$ satisfies the condition of Lipschitz with respect
to $x\in \mathfrak B$ with the constant of Lipschitz
$L<\frac{\nu}{2\mathcal N}.$
\end{enumerate}

Then
\begin{enumerate}
\item equation $(\ref{eqSL01})$ has the unique Lagrange stable
solution $\varphi$ with $\mathcal P\varphi(0)=0$ and \item
\begin{equation}\label{eqBF1}
\|\varphi -\varphi_{0}\|\le r ,
\end{equation}
where $$r:=\frac{4\mathcal N^2L\|f\|}{\nu (\nu -2\mathcal N L)}.$$
\end{enumerate}
\end{theorem}
\begin{proof}
Consider the space $\mathcal L(\mathbb R_{+},\mathfrak B)$
(respectively, $C_{b}(\mathbb R_{+},\mathfrak B)$). By Lemma
\ref{l3.4.1} it is a complete metric space. Define an operator
$$
\Phi :\mathcal L(\mathbb R_{+},\mathfrak B)\to \mathcal L(\mathbb
R_{+},\mathfrak B)
$$
(respectively
$$
\Phi : C_{b}(\mathbb R_{+},\mathfrak B)\to C_{b}(\mathbb
R_{+},\mathfrak B) )
$$
as follows: if $\psi\in \mathcal L(\mathbb R_{+},\mathfrak B)$
(respectively, $f\in C_{b}(\mathbb R_{+},\mathfrak B)$), then by
Lemma \ref{l3.4.2} $g\in \mathcal L(\mathbb R_{+},\mathfrak B)$,
where $g(t)=F(t,\psi(t))$ for any $t\in \mathbb R_{+}$. By Theorem
\ref{thLS11_2} the equation
\begin{equation}\label{eqD01}
\frac{dz}{dt}=Az+F(t,\psi(t)+\varphi_0(t))
\end{equation}
has the unique solution $\gamma\in \mathcal
L(\mathbb{R}_+,\mathfrak B)$ (respectively, $\gamma \in
C_{b}(\mathbb R_{+},\mathfrak B)$) and satisfies the condition
$\mathcal P\gamma(0)=0$. So, $\gamma \in \mathcal L(\mathbb
R_{+},\mathfrak B)$ (respectively, $\gamma \in C_{b}(\mathbb
R_{+},\mathfrak B)$). Let $\Phi(\psi):=\gamma$. From the said
above follows that $\Phi$ is well defined. Let us show that the
operator $\Phi$ is a contraction. In fact, it is easy to note that
the function $\gamma:=\gamma_1-\gamma_2=\Phi(\psi_1)-\Phi(\psi_2)$
is a solution of the equation
$$
\frac{du}{dt}=Au+F(t,\psi_1(t)+\varphi_0(t))-
                    F(t,\psi_2(t)+\varphi_0(t))
$$
with the initial condition $\mathcal P \gamma(0)=0$ and, by
Theorem \ref{thLS11_2}, it is subordinated to the estimate
$$
\begin{array}{c}
||\Phi(\psi_1)-\Phi(\psi_2)||\leq\\
\frac{2\mathcal N}{\nu}\sup\limits_{t\geq 0}
|F(t,\psi_1(t)+\varphi_0(t))-F(t,\psi_2(t)+\varphi_0(t))|\leq \\
\frac{2\mathcal N}{\nu}
L||\psi_1-\psi_2||=\alpha||\psi_1-\psi_2||.
\end{array}
$$
Since $\alpha=\frac{2\mathcal N}{\nu} L<\frac{2\mathcal N}{\nu}
\frac{\nu}{2\mathcal N}=1$, then $\Phi$ is a contraction and,
consequently, there exists the unique function $\overline{\psi}\in
\mathcal L(\mathbb R_{+},\mathfrak{B})$ (respectively, $\gamma \in
C_{b}(\mathbb R_{+},\mathfrak B)$) such that
$\Phi(\overline{\psi})=\overline{\psi}$. To finish the proof of
the theorem it is sufficient to assume that
$\varphi:=\overline{\psi}+\varphi_0$ and note that $\varphi$ is
desired solution.

To prove the second statement of Theorem we note that the function
$\psi :=\varphi -\varphi_{0}$ is a unique bounded on $\mathbb
R_{+}$ solution of equation (\ref{eqD01}) with $\mathcal
P\psi(0)=0$. By Theorem \ref{thLS11_2} we have
\begin{eqnarray}\label{eqD02}
& \|\varphi -\varphi_0\|=\|\psi \|\le \frac{2\mathcal
N}{\nu}\sup\limits_{t\ge 0}|F(t,\psi(t)+\varphi_0(t)|\le L\|\psi
+\varphi_0\|\le \nonumber \\
& \frac{2\mathcal N L}{\nu}(\|\psi\|+\|\varphi_0\|)\le
\frac{2\mathcal N L}{\nu}(r+\frac{2\mathcal N}{\nu}\|f\|)=r .
\end{eqnarray}
The theorem is proved.
\end{proof}

\begin{theorem}\label{th4A.1}
Let $f\in C(\mathbb{R},\mathfrak B)$, $F\in C(\mathbb{R}\times
\mathfrak B,\mathfrak B)$ and $A$ be an infinitesimal generator
which generates a $C_0$-semigroup $\{U(t)\}_{t\ge 0}$ acting on
$\mathfrak B$.

Assume that the following conditions are fulfilled:
\begin{enumerate}
\item[$1)$] the semigroup $\{U(t)\}_{t\ge 0}$ is hyperbolic;
\item[$2)$] the functions $f$ and $F$ are Lagrange stable
(respectively, $f$ and $F$ are bounded); \item[$3)$] $F$ satisfies
the condition of Lipschitz with respect to $x\in \mathfrak B$ with
the constant of Lipschitz $L<\frac{\nu}{2\mathcal N}$.
\end{enumerate}

Then equation $(\ref{eqSL1.1})$ has the unique Lagrange stable
solution $\varphi \in \mathcal L(\mathfrak L)$ (respectively,
bounded solution $\varphi \in C_{b}(\mathbb R,\mathfrak B)$.
\end{theorem}
\begin{proof} This statement can be proved using the same ideas as
in the proof of Theorem \ref{th4A}, but instead of Theorem
\ref{thLS11_2} we need to apply Theorem \ref{thLS11_3}.
\end{proof}

\begin{theorem}\label{thLS01} Under the conditions of Theorem \ref{th4A}
the unique Lagrange stable solution $\varphi$ of equation
(\ref{eqSL1.1}) with $\mathcal P\varphi(0)=0$ is compatible in the
limit, i.e., $\mathfrak L_{(f,F_{Q})}\subseteq \mathfrak
L_{\varphi}$.
\end{theorem}
\begin{proof}
Since the solution $\varphi$ is Lagrange stable, then the set
$Q:=\overline{\varphi(\mathbb R_{+})}$ is compact in $\mathfrak
B$. Let $\{t_n\}\in \mathfrak L_{(f,F_{Q})}$, where
$F_{Q}=F\big{|}_{\mathbb R_{+}\times Q}$, then the sequences
$\{f^{t_n}\}$ and $\{F^{t_n}_{Q}\}$ converge in $C(\mathbb
R_{+},\mathfrak B)$ and $C(\mathbb R_{+}\times Q,\mathfrak B)$.
Denote by $\tilde{f}=\lim\limits_{n\to \infty}f^{t_n}$ and
$\tilde{F}=\lim\limits_{n\to \infty}F_{Q}^{t_n}$.

Consider the sequence $\{\varphi^{t_n}\}$. It is pre-compact in
$C(\mathbb R_{+},\mathfrak B)$ because $\varphi$ is Lagrange
stable. We need to prove that the sequence $\{\varphi^{t_n}\}$
converges in $C(\mathbb R_{+},\mathfrak B)$. To this end it is
sufficient to show that it has at most one limiting point. Let
$\tilde{\varphi}$ be a limiting point of the sequence
$\{\varphi^{t_n}\}$, then there exists a subsequence
$\{t_{n_{k}}\}\subseteq \{t_n\}$ such that
$\tilde{\varphi}=\lim\limits_{k\to \infty}\varphi^{t_{n_{k}}}$.

Since $\varphi$ is a solution of equation (\ref{eqSL1.1}), then
the function $\varphi^{t_{n_{k}}}$ is a solution of equation
\begin{equation}\label{eqSL1.01}
x'=\mathcal A x+f^{t_{n_k}}(t) + F_{Q}^{t_{n_k}}(t,x)
\end{equation}
or equivalently it is a solution of linear equation
\begin{equation}\label{eqSL1.11}
x'=\mathcal A x+f^{t_{n_k}}(t) +
F_{Q}^{t_{n_k}}(t,\varphi^{t_{n_k}}(t)).
\end{equation}
Denote by $\mathfrak{f}(t):=f(t)+F(t,\varphi(t))$, then
$\mathfrak{f}^{t_{n_{k}}}=f^{t_{n_{k}}}(t)+F_{Q}^{t_{n_{k}}}(t,\varphi^{t_{n_{k}}}(t))$
for any $t\in\mathbb R_{+}$. By Lemma \ref{lLS1} the function
$\mathfrak{f}\in C(\mathbb R_{+},\mathfrak B)$ is Lagrange stable
and according to Lemma \ref{l3.4.2} (item (ii)) the sequence
$\{\mathfrak{f}^{t_{n_{k}}}\}$ converges in the space $C(\mathbb
R,\mathfrak B)$ to the function
$\tilde{\mathfrak{f}}(t):=\tilde{F}(t,\tilde{\varphi}(t))$ for any
$t\in\mathbb R$. By Lemma \ref{lLP1} (item (ii)) the function
$\tilde{\varphi}\in C(\mathbb R,\mathfrak B)$ is unique compact
solution ($\tilde{\varphi(\mathbb R)}\subseteq Q$) of equation
\begin{equation}\label{eqSL1.21}
x'=\mathcal Ax +\tilde{f}(t)+\tilde{F}(t,\tilde{\varphi}(t))
\end{equation}
or equivalently of equation
\begin{equation}\label{eqSL1.31}
x'=\mathcal Ax +\tilde{f}(t)+\tilde{F}(t,x).
\end{equation}
Note that by Theorem \ref{th4A.1} equation (\ref{eqSL1.31}) has a
unique Lagrange stable solution, because the function
$\tilde{F}\in H(F_{Q})$ and by Lemma \ref{lLC1} it is Lipschitzian
with the same constant $L$. Thus every limiting function
$\tilde{\varphi}\in C(\mathbb R,\mathfrak B)$ of the sequence
$\{\varphi^{t_n}\}$ is a Lagrange stable solution of equation
(\ref{eqSL1.31}). Since equation (\ref{eqSL1.31}) has a unique
Lagrange stable solution $\tilde{\varphi}\in C(\mathbb R,\mathfrak
B)$, then the sequence $\{\varphi^{t_n}\}$ converges, i.e.,
$\{t_n\}\in \mathfrak L_{\varphi}$. Theorem is proved.
\end{proof}

\begin{coro}\label{corSL1} Under the conditions of Theorem
\ref{th4A} if the functions $f$  and $F_{Q}$, where
$Q=\overline{\varphi(\mathbb R_{+})}$, are jointly asymptotically
stationary (respectively, asymptotically $\tau$-periodic,
asymptotically quasi-periodic with the spectrum of frequencies
$\nu_1,\ldots,\nu_k$, asymptotically Bohr almost periodic,
asymptotically Bohr almost automorphic, asymptotically Birkhoff
recurrent), then the solution $\varphi$ is so.
\end{coro}
\begin{proof} This statement follows from Theorems \ref{thSL0.1}, \ref{th4A} and \ref{thLS01}.
\end{proof}

\begin{theorem}\label{th4C}
Let $f\in C(\mathbb{R}_+,\mathfrak B)$, $F\in C(\mathbb{R}_+\times
\mathfrak B,\mathfrak B)$ and $A$ be an infinitesimal generator
which generates a $C_0$-semigroup $\{U(t)\}_{t\ge 0}$ acting on
$\mathfrak B$.

Assume that the following conditions are fulfilled:
\begin{enumerate}
\item[$1)$] the semigroup $\{U(t)\}_{t\ge 0}$ is hyperbolic;
\item[$2)$] the functions $f$ and $F$ are Lagrange stable;
\item[$3)$] $F$ is Lipschitzian with respect to the second
variable with the constant $L>0$.
\end{enumerate}

Then there exists a number $\varepsilon_0>0$ such that for every
$|\varepsilon|\leq \varepsilon_0$ equation
\begin{equation}\label{eq3.4.6C}
\frac{dx}{dt}=Ax+f(t)+\varepsilon F(t,x)
\end{equation}
has a unique compatible in the limit solution
$\varphi_{\varepsilon}\in C_{b}(\mathbb{R}_+,\mathfrak B)$
satisfying the condition $\mathcal P \varphi_{\varepsilon}(0)=0$.
Besides, the sequence $\{\varphi_{\varepsilon}\}$ converges to
$\varphi_0$ as $\varepsilon\to 0$ uniformly with respect to
$t\in\mathbb{R}_+$, where $\varphi_0 \in C_{b}(\mathbb
R_{+},\mathfrak B)$ is a unique compatible in the limit solution
of equation (\ref{eqSL1.1}) with $\mathcal P\varphi_0(0)=0$.
\end{theorem}
\begin{proof} Let $\varepsilon_{0}\in (0,\nu/2\mathcal N L)$.
Since the function $F$ is Lagrange stable, then the function
$\varepsilon F$ is so. Note the constant of Lipschits for
$\varepsilon F$ is less than $\varepsilon_{0}L$, because
$$
Lip(\varepsilon F)\le |\varepsilon|Lip(F)\le
\varepsilon_{0}L<\frac{\nu}{2\mathcal N}
$$
for any $|\varepsilon|\le \varepsilon_{0}$. According to Theorem
\ref{thLS01} for every $|\varepsilon|\leq\varepsilon_0$ equation
(\ref{eq3.4.6C}) has a unique compatible in the limit solution
$\varphi_{\varepsilon}\in C(\mathbb R_{+},Q_{r})$ satisfying the
condition $\mathcal P \varphi_{\varepsilon}(0)=0$.

Let us estimate the difference
$\varphi_{\varepsilon}(t)-\varphi_0(t)=\psi_{\varepsilon}(t)$. By
inequality (\ref{eqBF1}) we have
\begin{equation}\label{eqBF1.1}
\|\varphi_{\varepsilon}-\varphi_{0}\|\le \frac{4|\varepsilon|
\mathcal N^{2}L\|f\|}{\nu (\nu -2\mathcal N |\varepsilon|)}
\end{equation}
for any $\varepsilon \in (0,\varepsilon_{0})$. Passing to the
limit in inequality (\ref{eqBF1.1}) as $\varepsilon\to 0$, we get
the necessary statement. The theorem is proved.
\end{proof}

\begin{remark}\label{remC01} Note that Theorem \ref{th4C}
assures the existence at least one bounded on $\mathbb R_{+}$
solution $\varphi_{\varepsilon}$ of equation (\ref{eqSL1.1}) for
sufficiently small $\varepsilon$, but this equation can have on
the space $\mathfrak B$ more than one bounded on $\mathbb R_{+}$
solution. This fact we will confirm below by the corresponding
example.
\end{remark}

\begin{example}\label{exC1} Let $p\in C_{b}(\mathbb R_{+},\mathbb R)$ be a
positive function. Consider differential equation
\begin{equation}\label{eqC1}
x'=x-\varepsilon p(t)x^{3},
\end{equation}
where $\varepsilon \in \mathbb R_{+}$. For $\varepsilon =0$ admits
a unique bounded on $\mathbb R_{+}$ compatible in the limit
solution $\varphi_{0}(t)=0$ for all $t\in\mathbb R$. If
$\varepsilon
>0$, then equation (\ref{eqC1}) admits three bounded on $\mathbb R_{+}$ solutions:
$\varphi_{\varepsilon}^{1}(t)=0$,
$\varphi_{\varepsilon}^{2}(t)=q_{\varepsilon}(t)$ and
$\varphi_{\varepsilon}^{3}(t)=-q_{\varepsilon}(t)$ for all $t\in
\mathbb R_{+}$, where
$$
q_{\varepsilon}(t)=\varepsilon ^{-1/2}\Big{(}2\int\limits_{0}^{t}
e^{-2(t-\tau)}p(\tau)d\tau\Big{)}^{-1/2}\ \ (t\in \mathbb R_{+}).
$$
Note that $||\varphi_{\varepsilon}^{1}||\to 0$,
$||\varphi_{\varepsilon}^{2}||\to \infty$ and
$||\varphi_{\varepsilon}^{3}||\to \infty$ as $\varepsilon$ goes to
$0$.
\end{example}

\section{Hyperbolic sectorial operators}\label{S5}

Let $A$ be a linear closed operator with domain $D(A)$ and range
$R(A)$ in $\mathfrak B$. Denote by $\sigma(A)$ the spectrum of
operator $A$.

Recall that $\hat{\sigma}(A):=\sigma(A)\cup \{\infty\}$ is said to
be the extended spectrum of operator $A$. Note that
$\hat{\sigma}(A)\not= \emptyset$ and closed in the extended
complex plan $\hat{\mathbb C}:=\mathbb C \cup \{\infty\}$.

\begin{definition}\label{defSS1} A subset $\sigma$ of
$\hat{\sigma}(A)$is called (see, for example,
\cite[Ch.V]{Tay_1958}) a spectral set of $A$ if it is both open
and closed in the relative topology in the extended plane
$\hat{\mathbb C}$.
\end{definition}

\begin{theorem}\label{thSS1}\cite[Ch.V]{Tay_1958} Suppose $\hat{\sigma}(A)=\sigma_{1}\cup \ldots \cup
\sigma_{m}$ where $\sigma_{1},\ldots,\sigma_{m}$ are pairwise
disjoint spectral sets of $A$. Let $P_{j}$ ($j=1,\ldots,m$) be the
projection associated with $\sigma_{j}$ and let $X_{j}$ be the
range of $P_{j}$. Then $I=P_{1}+\ldots +P_{m}$, $P_{i}P_{j}=0$ if
$i\not= j$, $D(A)$ is invariant under $P_{j}$ and $X_{j}$ is
invariant under $A$.
\end{theorem}

\begin{theorem}\label{thSS2}\cite[Ch.V]{Tay_1958} Let $\sigma_{1}$ be a spectral set of
$A$, and let $A_{1}$ be the restriction of $A$ to the range
$X_{1}$ of $P_{1}$ (with $D(A_1)=D(A)\cap X_{1}$). Then
\begin{enumerate}
\item $\sigma_{1}=\hat{\sigma}(A_{1})$. \item If $\sigma_{1}$ does
not contain $\infty$, then $X_{1}\subset D_{n}(A):=\{x\in
\mathfrak B|\ x,Ax,\ldots,A^{n-1}x\in D(A)\}$ for each $n\ge 1$,
$A_{1}$ is continuous on $X_{1}$.
\end{enumerate}
\end{theorem}

\begin{definition}\label{defS2}
A closed operator $A$ with domain $D(A)$ that is dense in a Banach
space $\mathfrak B$ is called (see, for example,
\cite[Ch.I]{Hen_1981}) a sectorial operator\index{sectorial
operator} if for some $a\in \mathbb R $ and $\varphi \in
(0,\frac{\pi}{2})$ the sector
\begin{equation}\label{eqSO}
S_{a,\varphi}:=\{ \lambda \in \mathbb C,\ \varphi \le \vert
arg(\lambda - a)\vert \le \pi \}
\end{equation}
is contained in the resolvent set and for $\lambda \in
S_{a,\varphi} $
\begin{equation}
\Vert (\lambda I - A)^{-1}\Vert \le \frac{c}{\vert \lambda -
a\vert +1}.\nonumber
\end{equation}
\end{definition}

\begin{lemma}\label{lS1} Let $A$ be a sectorial operator and $\sigma(A)$ be its spectrum.
Then
$$
\mathbb M:=\hat{\sigma}(A)\cap i\mathbb R
$$
is a compact subset of
imaginary axis $i\mathbb R$.
\end{lemma}
\begin{proof}
Since $\mathbb M$ is the intersection of two closed sets, then it
is also closed. To finish the proof this statement it is
sufficient to note that $\mathbb M\subset [B,C]$, where $[B,C]$ is
the segment joining the points $B(0,|a|\sin \varphi)$ and
$C(0,-|a|\sin \varphi)$, $a\in \mathbb R$ and $0<\varphi <\pi/2$
from (\ref{eqSO}).
\end{proof}

\begin{definition}\label{defS3} A sectorial operator $A$ with
spectrum $\sigma(A)$ is said to be hyperbolic if $\sigma(A)\cap
i\mathbb R =\emptyset$.
\end{definition}

\begin{theorem}\label{thSO1} Let $A$ be a hyperbolic sectorial
operator and $\{e^{At}\}_{t\ge 0}$ be a semi-group of linear
operators generated by $A$. Then the following statements hold:
\begin{enumerate}
\item there exists a positive number $\alpha$ such that $|Re\
\lambda|\ge \alpha$ for any $\lambda \in \sigma(A)$; \item
$\sigma_{-}(A):=\sigma(A)\cap \{z\in \mathbb C|\ Re\ z <0\}$ is a
bounded spectral subset; \item $A_{-}:=P_{-}A$, where $P_{-}$
(respectively, $P_{+}$) is the spectral projection associated with
$\sigma_{-}(A)$ (respectively,
$\hat{\sigma}_{+}(A):=\hat{\sigma}(A)\setminus \sigma_{-}(A)$), is
a linear continuous on $\mathfrak B_{-}:=P_{-}(\mathfrak B)$
operator (respectively, is a sectorial operator on $\mathfrak
B_{+}:=P_{+}(\mathfrak B)$) and $\sigma(A_{-})=\sigma_{-}(A)$
(respectively, $\hat{\sigma}(A_{+})=\hat{\sigma}_{+}(A)$); \item
$Re\ \sigma(A_{-}) \le -\alpha$ (respectively, $Re\
\hat{\sigma}(A_{+})>\alpha$); \item there are positive constants
$\mathcal N$ and $\nu$ such that
\begin{equation}\label{eqSO1}
\| e^{A_{-}t}\|\le \mathcal N e^{-\nu t}\nonumber
\end{equation}
for any $t\in\mathbb R_{+}$; \item $A_{+}:=P_{+}A$ is a sectorial
operator, $\hat{\sigma}(A_{+})=\hat{\sigma}_{+}(A)$,
\begin{equation}\label{eqSO2}
\|e^{-A_{+}t}\|\le \mathcal N e^{-\nu t}\ \ \mbox{and}\ \
\|A_{+}e^{-A_{+}t}\|\le \mathcal N t^{-1}e^{-\nu t}\nonumber
\end{equation}
for any $t >0$.
\end{enumerate}
\end{theorem}
\begin{proof}
If we suppose that the first statement is false, then there exists
a sequence $\{\lambda_{n}\}\subseteq \hat{\sigma}(A)$ such that
$|Re\ \lambda_{n}|\to 0$ as $n\to \infty$ and, consequently, there
exists a positive number $b$ such that $|Re \ \lambda_{n}|\le b$
for any $n\in\mathbb N$. Note that the sequence $\{\lambda_{n}\}$
is a subset of rectangle with vertices $A_{1}(b,(|a|+b)\sin
\varphi)$, $A_{2}(b,-(|a|+b)\sin \varphi)$, $A_{3}(a,-(|a|+b)\sin
\varphi)$ and $A_{4}(a,(|a|+b)\sin \varphi)$. This means that the
sequence $\{\lambda_{n}\}$ is bounded and without loss of
generality we can assume that $\{\lambda_{n}\}$ is a convergent
sequence. Denote by $\lambda_{0}:=\lim\limits_{n\to
\infty}\lambda_{n}$, then $\lambda_{0}\in \hat{\sigma}(A)$ and
$Re\ \lambda_{0}=0$ hence $\lambda_{0}\in \sigma(A)\cap i\mathbb
R\subseteq M=\hat{\sigma}(A)\cap i\mathbb R$. The last relation
contradicts to our assumption that $\sigma(A)\cap i\mathbb
R=\emptyset$. The obtained contradiction proves our statement.

Note that the set $\sigma_{-}(A)$ is bounded because it is a
subset of triangle $\Delta_{B_1B_2B_3}$ with vertices
$B_{1}(-\alpha,|a|\sin \varphi)$, $B_{2}(-\alpha,-|a|\sin
\varphi)$ and $B_{3}(a,0)$. Now we will show that the set
$\sigma_{-}(A)$ is closed. Let $\lambda$ belongs to closure of the
set $\sigma_{-}(A)$ then there exists a sequence
$\{\lambda_{k}\}\subset \sigma_{-}(A)$ such that $\lambda_{k}\to
\lambda_{0}$ as $k\to \infty$. Since the set $\sigma(A)\subset
\hat{\sigma}(A)$ then $\lambda_{0}\in \hat{\sigma}(A)$ because
$\hat{\sigma}(A)$ is closed. According to the second statement $Re
\ \lambda_{k}\le -\alpha$ for any $k\in \mathbb N$ and,
consequently, $Re \ \lambda_{0}\le -\alpha$. Since
$\hat{\sigma}(A)=\sigma_{-}(A)\coprod \sigma_{+}(A)\cup
\{\infty\}$ and $Re \ \lambda_{0}\le -\alpha$ we concludes that
$\lambda_{0}\in \sigma_{-}(A)$. Finally we remains to prove that
the set $\hat{\sigma}_{+}(A)=\hat{\sigma}(A)\setminus
\sigma_{-}(A)$ is closed. Let $\lambda_{0}$ belongs to the closure
of $\hat{\sigma}(A)\setminus \sigma_{-}(A)$. Logically two cases
are possible:
\begin{enumerate}
\item $\lambda_{0}=\infty$, then evidently it belongs to
$\hat{\sigma}(A)\setminus \sigma_{-}(A)=\sigma_{+}(A)\cup
{\infty}$; \item $\lambda_{0}\not= \infty$, then $\lambda_{0}\in
\mathbb C$ and there exists a sequence $\{\lambda_{k}\}\subseteq
\sigma_{+}(A)$ such that $\lambda_{0}=\lim\limits_{k\to
\infty}\lambda_{k}$. Since $\lambda_{k}\in \sigma_{+}(A)$ then $Re
\ \lambda_{k}\ge \alpha$ for any $k\in\mathbb N$ and,
consequently, $Re \ \lambda_{0}\ge \alpha$. Since
$\hat{\sigma}(A)=\sigma_{-}(A)\coprod \sigma_{+}(A)\cup
\{\infty\}$, $\{\lambda_{k}\}\subset \sigma_{+}(A)\subset
\hat{\sigma}(A)$ and $\hat{\sigma}(A)$ is closed, then
$\lambda_{0}\in \hat{\sigma}(A)\cap \mathbb C$ and $Re \
\lambda_{0}\ge \alpha$ and, consequently, $\lambda_{0}\in
\sigma_{+}(A)\subseteq \hat{\sigma}(A)\setminus \sigma_{-}(A)$.
Thus the set $\hat{\sigma}_{+}(A)$ is a spectral set.
\end{enumerate}

The third statement follows from the second one and Theorems
\ref{thSS1} and \ref{thSS2}.

The fourth statement follows from the first and third statements.

The fifth statement follows from the third statement and Theorem
4.1 from \cite[Ch.I]{Dal}.

The sixth statement follows from the third statement and Theorem
1.5.3 from \cite[Ch.I]{Hen_1981}.
\end{proof}

\begin{remark}\label{remP1} 1. $AP_{\pm}x=P_{\pm}Ax$ for any $x\in
D(A)$.

2. The operator $A_{-}$ is defined only on $D(A)$, but taking into
account that $D(A)$ is dense on $\mathfrak B$ the operator $A_{-}$
may be extended by continuity on $\mathfrak B$.
\end{remark}

\begin{coro}\label{corSO1} Let $-A$ be a sectorial hyperbolic
operator. Then the following statements hold:
\begin{enumerate}
\item there exists a positive number $\alpha$ such that $|Re\
\lambda|\ge \alpha$ for any $\lambda \in \sigma(A)$; \item
$\sigma_{+}(A):=\sigma(A)\cap \{z\in \mathbb C|\ Re\ z >0\}$ is a
bounded spectral subset; \item $A_{+}:=P_{+}A$, where $P_{+}$
(respectively, $P_{-}$) is the spectral projection associated with
$\sigma_{+}(A)$ (respectively,
$\hat{\sigma}_{-}(A):=\hat{\sigma}(A)\setminus \sigma_{+}(A)$), is
a linear continuous on $X_{+}:=P_{+}(X)$ operator (respectively,
is a sectorial operator on $X_{-}:=P_{-}(X)$) and
$\sigma(A_{+})=\sigma_{+}(A)$ (respectively,
$\hat{\sigma}(A_{-})=\hat{\sigma}_{-}(A)$); \item $Re\
\sigma(A_{+}) \ge \alpha$ (respectively, $Re\
\hat{\sigma}(A_{-})\le -\alpha$); \item there are positive
constants $\mathcal N$ and $\nu$ such that
\begin{equation}\label{eqSO1.1}
\| e^{-A_{+}t}\|\le \mathcal N e^{-\nu t}
\end{equation}
for any $t\in\mathbb R_{+}$; \item
$\hat{\sigma}(A_{-})=\hat{\sigma}_{-}(A)$, $-A_{-}$ is a sectorial
operator with $Re\ \sigma(-A_{-})\ge \alpha$ and there exist
positive constants $\mathcal N$ and $\nu$ such that
\begin{equation}\label{eqSO2.1}
\|e^{A_{-}t}\|\le \mathcal N e^{-\nu t}\ \ \mbox{and}\ \
\|A_{-}e^{A_{-}t}\|\le \mathcal N  t^{-1}e^{-\nu t}
\end{equation}
for any $t >0$.
\end{enumerate}
\end{coro}
\begin{proof} Note that $\sigma(A)=-\sigma(-A)$ and by Theorem
\ref{thSO1} (item (i)) there exists a positive number $\alpha >0$
 such that $|Re\ \lambda|\ge \alpha$ for any $\lambda \in
 \sigma(A)$.

 On the other hand $\sigma(A)=\sigma_{-}(A)\cup \sigma_{+}(A)$ and
 $A=(A_{-},A_{+})$, where $A_{\pm}=P_{\pm}(A)$. Since
 $-A=(-A_{+},-A_{-})$ (that is  $(-A)_{-}=-A_{+}$ and
 $(-A)_{+}=-A_{-}$) then by Theorem \ref{thSO1}, items (v)
 (respectively, Theorem \ref{thSO1}, item(vi)) there exist
 positive constants $\mathcal N$ and $\nu$ such that
 (\ref{eqSO1.1}) (respectively, (\ref{eqSO2.1})) takes place.
\end{proof}

\begin{lemma}\label{lC01} Let $A:D(A)\to \mathfrak B$ be the
infinitesimal generator of $C_0$-semigroup $\{U(t)\}_{t\ge 0}=
\{e^{At}\}_{t\ge 0}$, $\sigma \subset \sigma(A)$ be a
spectral set of $A$ and $P$ be the projection associated with
$\sigma$. Then the following statements hold:
\begin{enumerate}
\item $PA=AP$;  \item $Pe^{At}P=e^{PAt}P$ for any $t\in\mathbb
R_{+}$; \item $Pe^{At}=e^{At}P$ for any $t\in\mathbb R_{+}$.
\end{enumerate}
\end{lemma}
\begin{proof}
The first statement directly follows from the operational calculus
for unbounded closed operators (see, for example,
\cite[Ch.VII]{DS_1966} and \cite[Ch.V]{Tay_1958}).


Let $x_0\in \mathfrak B$. Consider the functions
$u(t,x_0):=Pe^{At}Px_0$ (respectively, $u(t,x_0):=Pe^{At}x_0$) and
$v(t,x_0))=e^{PAt}Px_0$ (respectively $v(t,x_0)=e^{At}Px_0$) for
any $t\in\mathbb R_{+}$. If $x\in D(A^2)$, then the functions $u$
and $v$ are continuously differentiable solutions of Cauchy
problem \cite[Ch.I]{Kre_1963}
\begin{equation}\label{eqC_01}
x'=Ax,\ \ x(0)=Px_0\nonumber
\end{equation}
and, consequently,
\begin{equation}\label{eqC_02}
u(t,x_0)=v(t,x_0)
\end{equation}
for any $x_0\in D(A^2)$ and $t\ge 0$.

Let now $x_0\in \mathfrak B$. Since $D(A^2)$ is dense in
$\mathfrak B$ (see, for example, \cite[Ch.I]{Hen_1981} or
\cite[Ch.I]{Kre_1963}), then there exists a sequence
$\{x^{k}_{0}\}\subset D(A^{2})$ such that $x_{0}^{k}\to x_0$ as
$k\to \infty$. By (\ref{eqC_02}) we have
\begin{equation}\label{eqC_03}
u(t,x_0^{k})=v(t,x_0^{k})
\end{equation}
for any $k\in\mathbb N$ and $t\ge 0$. Passing to the limit in
(\ref{eqC_03}) and taking into consideration that the operators
$Pe^{At}P$ (respectively, $Pe^{At}$) and $e^{PA t}P$
(respectively, $e^{At}P$) acting on the space $\mathfrak B$ are
continuous we obtain
\begin{equation}\label{eqC_04}
u(t,x_0)=v(t,x_0)\nonumber
\end{equation}
for any $x_0\in \mathfrak B$ and $t\ge 0$, i.e., $Pe^{At}P=e^{PA
t}P$ (respectively, $Pe^{At}=e^{At}P$) for any $t\ge 0$. Lemma is
completely proved.
\end{proof}

\begin{coro}\label{corC1} Let $A:D(A)\to \mathfrak B$ be a
hyperbolic sectorial operator and $P_{\pm}$ be the spectral
projection of $A$. Then the following statements hold:
\begin{enumerate}
\item $P_{\pm}A=AP_{\pm}$; \item $P_{\pm}e^{At}=e^{At}P_{\pm}$ for
any $t\in\mathbb R_{+}$; \item
$P_{\pm}e^{At}P_{\pm}=e^{P_{\pm}At}P_{\pm}$ for any $t\in\mathbb
R_{+}$.
\end{enumerate}
\end{coro}

Let $\mathcal A: D(\mathcal A)\to \mathfrak B$ be a linear closed
operator acting from $D(\mathcal A)$ to $\mathfrak B$ generating a
$C_0$-semigroup $\{U(t)\}_{t\ge 0}$ (notation $\{U(t)\}_{t\ge
0}=\{e^{\mathcal A t}\}_{t\ge 0}$) of linear bounded operators on
the space $\mathfrak B$.

\begin{lemma}\label{lC_01} Assume that $-\mathcal A$ is a
sectorial hyperbolic operator, then the following statements hold:
\begin{enumerate}
\item the semigroup $\{U(t)\}_{t\ge 0}$ is hyperbolic; \item a
Green's function $G(t)$ for hyperbolic semigroup $U(t)$ is defined
by
\begin{equation}\label{eqLS2_1}
 G(t):=\left\{\begin{array}{ll}
&\!\! e^{A_{-}t} P_{-},\;\ \ \ \ \ \mbox{if}\;\ t\ge 0 \\[2mm]
&\!\! -e^{-A_{+}t}P_{+}, \;\ \mbox{if}\;\ t<0\ .
\end{array}
\right.\nonumber
\end{equation}
\end{enumerate}
\end{lemma}
\begin{proof} Since $\sigma(A)=-\sigma(-A)$, then by Theorem
\ref{thSO1} we have:
\begin{enumerate}
\item there exists a positive number $\alpha$ such that $|Re\
\lambda|\ge \alpha$ for any $\lambda \in \sigma(A)$; \item
$\sigma_{+}(\mathcal A)=-\sigma_{-}(-\mathcal A)$ is a bounded
spectral subset; \item $A_{+}:=P_{+}A$, where $P_{+}$ is the
spectral projection associated with $\sigma_{+}(A)$, is a linear
continuous on $\mathfrak B_{+}:=P_{+}(\mathfrak B)$ operator with
$Re\ \sigma(A_{+})\le -\alpha$; \item there are positive constants
$\mathcal N$ and $\nu$ such that
\begin{equation}\label{eqS0O1}
\| e^{-A_{+}t}\|\le \mathcal N e^{-\nu t}\nonumber
\end{equation}
for any $t\in\mathbb R_{+}$; \item
$\sigma_{-}(A):=-\sigma_{+}(-A)$ and
$\hat{\sigma}_{-}(A):=-\hat{\sigma}_{+}(-A)$; \item $Re\
\hat{\sigma}(A_{-})>\alpha$; \item $A_{+}:=P_{+}A,\
A_{+}=-(-A)_{-}$ and $(-A)_{-}$ is a sectorial operator with $Re\
\sigma(-A)_{-}\ge \alpha$,
$\hat{\sigma}(A_{+})=-\hat{\sigma}_{+}(-A)$ and there exist
positive constants $\mathcal N$ and $\nu$ such that
\begin{equation}\label{eqSO02}
\|e^{A_{-}t}\|\le \mathcal N e^{-\nu t}\ \ \mbox{and}\ \
\|A_{-}e^{A_{-}t}\|\le \mathcal N t^{-1}e^{-\nu t}\nonumber
\end{equation}
for any $t >0$.
\end{enumerate}

Let $\mathcal P=P_{-}$ and $\mathcal Q=I-\mathcal
P=I-P_{-}=P_{+}$, then
\begin{equation}\label{eqGr1}
U(t)\mathcal P=e^{At}P_{-}=e^{A_{-}t}P_{-}
\end{equation}
and from
(\ref{eqSO02}) we have $\|U(t)\mathcal P\|\le \mathcal N e^{-\nu
t}$ for any $t\ge 0$. On the other hand if $t<0$, then
\begin{equation}\label{eqGr2}
U_{\mathcal Q}(t)=[U(-t)\mathcal Q]^{-1}=[e^{\mathcal
A(-t)}P_{+}]^{-1}=[e^{A_{+}(-t)}P_{+}]^{-1}=e^{A_{+}t}P_{+}
\end{equation}
and by (\ref{eqS0O1}) we have $\|U_{\mathcal Q}(t)\|\le \mathcal N
e^{\nu t}$ for any $t<0$. Thus the semigroup $\{U(t)\}_{t\ge 0}$
is hyperbolic.

The second statement of Lemma follows from (\ref{eqGr1}) and
(\ref{eqGr2}). Lemma is completely proved.
\end{proof}

\begin{lemma}\label{lC02}
Let $-A$ be a hyperbolic sectorial operator acting in $\mathfrak
B$ with domain $D(A)$ and $f\in C_{b}(\mathbb R_{+},\mathfrak B)$.
Then equation (\ref{eqSL1.1}) has a unique solution $\varphi \in
C_{b}(\mathbb R_{+},\mathfrak B)$ with $P_{-}\varphi(0)=0$ and
\begin{equation}\label{eqC02}
\varphi(t)= \int_{0}^{+\infty}G(t-\tau)f(\tau)d\tau
=\int_{0}^{t}e^{A_{-}(t-\tau)}f_{-}(\tau)d\tau -
\int_{t}^{+\infty}e^{A_{+}(t-\tau)}f_{+}(\tau)d\tau .\nonumber
\end{equation}
\end{lemma}
\begin{proof} This statement follows from Theorem \ref{thLS11_2} (item $(i)$) and Lemma \ref{lC_01}.
\end{proof}

\section{Averaging Principle on Semi-Axis for Semi-Linear
Differential Equations with Stationary Linear Part}\label{S6}

\subsection{Linear Differential Equations}\label{S6.1}

Let $\varepsilon_{0}$ be some fixed positive number. Consider the
equation
\begin{equation}\label{eqALN1}
x'=\varepsilon (\mathcal A x+f(t)),
\end{equation}
where $\mathcal A : D(\mathcal A)\to \mathfrak B$, $f\in C(\mathbb
R_{+},\mathfrak B)$ and $0<\varepsilon \le \varepsilon_0$.

\begin{definition}\label{defAL1}\rm
Let $f:\mathbb R_{+}\times (0,\varepsilon_0]\to \mathfrak B$.
Following \cite[Ch.IV]{KBK} we say that {\em $f(t;\varepsilon)$
integrally converges to $0$} if for any $L>0$ we have
\begin{equation}\label{eqAL1}
\lim\limits_{\varepsilon \to 0}\sup\limits_{|t-s|\le L,\ t,s\in
\mathbb R_{+}} |\int_{s}^{t}f(\tau;\varepsilon)d\tau |=0.\nonumber
\end{equation}
If additionally there exists a constant $m>0$ such that
\begin{equation}\label{eqAL2}
|f(t;\varepsilon)|\le m \ \nonumber
\end{equation}
for any $t\in\mathbb R_{+}$ and $0<\varepsilon\le
\varepsilon_{0}$, then we say that $f(t;\varepsilon)$ {\em
correctly converges to} $0$ as $\varepsilon \to 0$.
\end{definition}

\begin{lemma}\label{lL1_0}\cite[ChIV]{KBK} \rm
Let $\mathbb T =\mathbb R$ or $\mathbb R_{+}$. If $f\in C(\mathbb
T,\mathfrak B)$ and
\begin{equation}\label{AL3}
\lim\limits_{T\to+\infty}\frac{1}{T} \left|\int_{t}^{t+T}f(s)d
s\right|=0\nonumber
\end{equation}
uniformly with respect to $t\in\mathbb T$, then
$f(t;\varepsilon):=f(\frac{t}{\varepsilon})$ integrally converges
to $0$ as $\varepsilon \to 0$. If additionally the function $f$ is
bounded on $\mathbb T$, then $f(t;\varepsilon)$ correctly
converges to $0$ as $\varepsilon \to 0$.
\end{lemma}

Note that Lemma \ref{lL1_0} is proved in \cite[ChIV]{KBK} for
$\mathbb T=\mathbb R$. If $\mathbb T=\mathbb R_{+}$ this statement
can be proved using absolutely the same arguments as in the case
$\mathbb T=\mathbb R$.

Below we will use the following condition (\textbf{A}):

$f\in C(\mathbb R_{+},\mathfrak B)$ and there exists $\bar{f}\in
\mathfrak B$ such that
\begin{equation}\label{eqAL5}
\lim\limits_{T\to +\infty}\frac{1}{T} \left |\int_{t}^{t+T}
[f(s)-\bar{f}]d s\right |=0
\end{equation}
uniformly with respect to $t\in\mathbb R_{+}$.

Denote by $\Psi$ the family of all decreasing, positive bounded
functions $\psi :\mathbb R_{+}\to\mathbb R_{+}$ with
$\lim\limits_{t\to +\infty}\psi(t)=0$.

\begin{lemma}\label{lC1}\cite{CL_2021} Let $l>0$ and $\psi\in \Psi$, then
\begin{equation*}
\lim\limits_{\varepsilon \to 0}\sup\limits_{0\le \tau\le
l}\tau\psi(\frac{\tau}{\varepsilon})=0 .
\end{equation*}
\end{lemma}

\begin{remark}\label{remA1} \rm
By Lemma 2 in \cite{CD_2004} equality (\ref{eqAL5}) holds if and
only if there exists a function $\omega \in\Psi$ satisfying
\begin{equation}\label{AL5.1}
\frac{1}{T}\left |\int_{t}^{t+T} [f(s)-\bar{f}]d s\right |
\le\omega(T)\nonumber
\end{equation}
for any $T>0$ and $t\in\mathbb R_{+}$.
\end{remark}

\begin{lemma}\label{lG_0} Let $\alpha >0$, $\omega\in \Psi$ and $\varepsilon \in
(0,\alpha]$. Then the following statements hold:
\begin{enumerate}
\item
\begin{eqnarray}\label{eqlG0.1}
\lim\limits_{\varepsilon \to 0^{+}}\sup\limits_{t\in\mathbb R_{+}}
e^{-\nu t}t\omega(\frac{t}{\varepsilon})=0;
\end{eqnarray}
\item
\begin{eqnarray}\label{eqlG0.2}
\lim\limits_{\varepsilon \to 0^{+}}\sup\limits_{t\in\mathbb R_{+}}
\int_{0}^{t}e^{-\nu s}s\omega(\frac{s}{\varepsilon})ds =0;
\end{eqnarray}
\item
\begin{eqnarray}\label{eqlG0.3}
\lim\limits_{\varepsilon \to 0^{+}}\int_{0}^{+\infty}e^{-\nu
s}s\omega(\frac{s}{\varepsilon})ds =0 .
\end{eqnarray}
\end{enumerate}
\end{lemma}
\begin{proof}
Let $\mu \in (1,+\infty)$ and $\xi :=\varepsilon^{\mu}$. Note that
\begin{eqnarray}\label{eqlG0.4}
& e^{-\nu t}t\omega(\frac{t}{\varepsilon})\le \sup\limits_{0\le
t\le \varepsilon^{\mu}}e^{-\nu t}t\omega(\frac{t}{\varepsilon})+
\sup\limits_{ t\ge \varepsilon^{\mu}}e^{-\nu
t}t\omega(\frac{t}{\varepsilon})\le \nonumber \\
& \omega(0)\varepsilon^{\mu} +\omega(\varepsilon^{\mu
-1})\max\{1;\frac{1}{\nu e}\}\nonumber
\end{eqnarray}
and, consequently
\begin{eqnarray}\label{eqlG0.5}
\sup\limits_{t\in\mathbb R_{+}} e^{-\nu
t}t\omega(\frac{t}{\varepsilon})\le \omega(0)\varepsilon^{\mu}
+\omega(\varepsilon^{\mu -1})\max\{1;\frac{1}{\nu e}\}
\end{eqnarray}
for any $\varepsilon \in (0,\alpha]$. Passing to the limit in
(\ref{eqlG0.5}) as $\varepsilon \to 0$ we obtain (\ref{eqlG0.1}).

To prove (\ref{eqlG0.2}) we note that
\begin{equation}\label{eqlG0.6}
\int_{0}^{t}e^{-\nu \tau}\tau
\omega(\frac{\tau}{\varepsilon})d\tau =
\int_{0}^{\varepsilon^{\mu}} e^{-\nu \tau}\tau
 \omega(\frac{\tau}{\varepsilon})d\tau +
\int_{\varepsilon^{\mu}}^{t}e^{-\nu \tau}\tau
 \omega(\frac{\tau}{\varepsilon})d\tau .\nonumber
\end{equation}
Taking into consideration that
\begin{eqnarray}\label{eqlG0.7}
&  \int_{0}^{\varepsilon^{\mu}}e^{-\nu
\tau}\tau \omega(\frac{\tau}{\varepsilon})d\tau\le \nonumber \\
& \omega(0) \int_{0}^{\varepsilon^{\mu}}\tau e^{-\nu \tau}d\tau =
\frac{\omega(0)}{\nu}[\frac{1}{\nu} - e^{-\nu
\varepsilon^{\mu}}(\varepsilon^{\mu}+\frac{1}{\nu}) ]\to 0
\end{eqnarray}
and
\begin{eqnarray}\label{eqlG0.8}
& \int_{\varepsilon^{\mu}}^{t}e^{-\nu \tau}\tau
\omega(\frac{\tau}{\varepsilon})d\tau\le
\omega(\frac{\varepsilon^{\mu}}{\varepsilon})\int_{\varepsilon^{\mu}}^{t}\tau
e^{-\nu
s}ds =\nonumber \\
& \omega(\varepsilon^{\mu -1})\Big{(}-\frac{e^{-\nu
t}}{\nu}(t+\frac{1}{\nu})+\frac{e^{-\nu
\varepsilon^{\mu}}}{\nu}(\varepsilon^{\mu}+\frac{1}{\nu})\Big{)}\le
\nonumber \\
& \omega(\varepsilon^{\mu -1})\frac{e^{-\nu
\varepsilon^{\mu}}}{\nu}(\varepsilon^{\mu}+\frac{1}{\nu})\to 0
\end{eqnarray}
as $\varepsilon \to 0$ uniformly with respect to $t\in \mathbb
R_{+}$ from (\ref{eqlG0.6})-(\ref{eqlG0.8}) we obtain
(\ref{eqlG0.2}).

Finally, we will estimate the integral $\int_{0}^{+\infty}e^{-\nu
s}\omega(\frac{s}{\varepsilon})ds$. To this end we note that
\begin{eqnarray}\label{eqlG0.9}
& \int_{0}^{+\infty}e^{-\nu s}s\omega(\frac{s}{\varepsilon})ds=
\int_{0}^{\varepsilon^{\mu}}e^{-\nu
s}s\omega(\frac{s}{\varepsilon})ds + \nonumber \\
& \int_{\varepsilon^{\mu}}^{+\infty}e^{-\nu
s}s\omega(\frac{s}{\varepsilon})ds \le \nonumber \\
& \omega(0)\int_{0}^{\varepsilon^{\mu}}e^{-\nu s}sds +
\omega(\varepsilon^{\mu
-1})\int_{\varepsilon^{\mu}}^{+\infty}e^{-\nu s}sds .
\end{eqnarray}
Since
\begin{equation}\label{eqlG0.10}
\int e^{-\nu s}sds=-\frac{1}{\nu}e^{-\nu s}(s+\frac{1}{\nu}),
\end{equation}
then from (\ref{eqlG0.9})-(\ref{eqlG0.10}) we have
\begin{eqnarray}\label{eqlG0.11}
& \int_{0}^{+\infty}e^{-\nu s}s\omega(\frac{s}{\varepsilon}) ds
\le
\\
& \omega(0)[-\frac{1}{\nu}e^{-\nu
\varepsilon^{\mu}}(\varepsilon^{\mu}+\frac{1}{\nu})+\frac{1}{\nu^{2}}]+\omega(\varepsilon^{\mu
-1})\frac{1}{\nu}e^{\nu
\varepsilon^{\mu}}(\varepsilon^{\mu}+\frac{1}{\nu}) .\nonumber
\end{eqnarray}
Passing to the limit in (\ref{eqlG0.11}) as $\varepsilon \to 0$ we
obtain (\ref{eqlG0.3}). Lemma is completely proved.
\end{proof}

\begin{lemma}\label{lG1}
Assume that the following conditions hold:
\begin{enumerate}
\item $f\in C_{b}(\mathbb R_{+},\mathfrak B)$; \item there exists
an element $\bar{f}\in \mathfrak B$ such that
\begin{equation}\label{eqA1}
\lim\limits_{T\to
+\infty}\frac{1}{T}\int_{t}^{t+T}[f(s)-\bar{f}]ds=0\nonumber
\end{equation}
uniformly with respect to $t\in\mathbb R_{+}$; \item $-\mathcal A$
is a sectorial hyperbolic operator.
\end{enumerate}

Then
\begin{equation}\label{eqA2}
\lim\limits_{\varepsilon \to 0^{+}}\|\varphi_{\varepsilon}\|=0
,\nonumber
\end{equation}
where
\begin{equation}\label{eqA3}
\varphi_{\varepsilon}(t)=\int_{0}^{+\infty}G(t-\tau)\hat{f}_{\varepsilon}(\tau)d\tau
\nonumber
\end{equation}
and $\hat{f}_{\varepsilon}(t);=f(\frac{t}{\varepsilon})-\bar{f}$
for any $t\in\mathbb R_{+}$ and $\varepsilon \in (0,\alpha]$.
\end{lemma}
\begin{proof} We will show that
\begin{equation}\label{eqA4}
\lim\limits_{\varepsilon \to
0}\|\varphi_{\varepsilon}\|=\lim\limits_{\varepsilon \to
0}\sup\limits_{t\in\mathbb
R_{+}}|\int_{0}^{+\infty}G(t-\tau)\hat{f}_{\varepsilon}(\tau)d\tau|=0
.\nonumber
\end{equation}
To this end we note that
\begin{eqnarray}\label{eqA5}
& |\int_{t}^{t+s}\hat{f}_{\varepsilon}(\tau)d\tau | = \varepsilon
|\int_{t/\varepsilon}^{t/\varepsilon+s/\varepsilon}[f(\tau)-\bar{f}]d\tau
| \le \nonumber \\
& s\omega(\frac{s}{\varepsilon})\le \sup\limits_{0\le s\le
l}s\omega(\frac{s}{\varepsilon})\nonumber
\end{eqnarray}
and, consequently,
\begin{equation}\label{eqA6}
\sup\limits_{0\le s\le l}\sup\limits_{t\in\mathbb
R_{+}}|\int_{t}^{t+s}\hat{f}_{\varepsilon}(\tau)d\tau \le
\sup\limits_{0\le s\le l}s\omega(\frac{s}{\varepsilon}) \to
0\nonumber
\end{equation}
as $\varepsilon \to 0^{+}$ for any $l>0$, because by Lemma
\ref{lC1} $\sup\limits_{0\le s\le l}s\omega(\frac{s}{\varepsilon})
\to 0$ as $\varepsilon \to 0^{+}$.

By Lemma \ref{lC02} we have
\begin{equation}\label{eqA7}
|\varphi_{\varepsilon}(t)|=
|\int_{0}^{+\infty}G(t-\tau)\hat{f}_{\varepsilon}(\tau)d\tau|=
\end{equation}
$$
|\int_{0}^{t}e^{A_{-}(t-\tau)}P_{-}\hat{f}_{\varepsilon}(\tau)d\tau
-
\int_{t}^{+\infty}e^{A_{+}(t-\tau)}P_{+}\hat{f}_{\varepsilon}(\tau)d\tau|=
$$
$$
|-\int_{-t}^{0}e^{-A_{-}s}P_{-}\hat{f}_{\varepsilon}(t+s)ds -
\int_{0}^{+\infty}e^{-A_{+}s}P_{+}\hat{f}_{\varepsilon}(t+s)ds|\le
$$
$$
|\int_{-t}^{0}e^{-A_{-}s}P_{-}\hat{f}_{\varepsilon}(t+s)ds|+
|\int_{0}^{+\infty}e^{-A_{+}s}P_{+}\hat{f}_{\varepsilon}(t+s)ds|=I_{1}(t,\varepsilon)+I_{2}(t,\varepsilon),
$$
where
\begin{equation}\label{eq8}
I_{1}(t,\varepsilon)=|\int_{-t}^{0}e^{-A_{-}s}P_{-}\hat{f}_{\varepsilon}(t+s)ds|\nonumber
\end{equation}
and
\begin{equation}\label{eq9}
I_{2}(t,\varepsilon)=|\int_{0}^{+\infty}e^{-A_{+}s}P_{+}\hat{f}_{\varepsilon}(t+s)ds|.\nonumber
\end{equation}
Integrating by parts in s we find
\begin{eqnarray}\label{eqHI4}
& \int_{-t}^{0}e^{-A_{-}s}P_{-}\hat{f}_{\varepsilon}(t+s)ds =
\nonumber
\\
&
e^{A_{-}t}P_{-}\int_{0}^{t}\hat{f}_{\varepsilon}(s)ds+\int_{-t}^{0}A_{-}e^{-A_{-}s}P_{-}\int_{t}^{t+s}\hat{f}_{\varepsilon}(\sigma)d\sigma
ds .\nonumber
\end{eqnarray}
Notice that
\begin{equation}\label{eqH14.1}
|\int_{t}^{t+s}\hat{f}_{\varepsilon}(\sigma)d\sigma|\le As
\end{equation}
for any $t,s\in \mathbb R_{+}$, where $A:= 2\|f\|$.

\begin{eqnarray}\label{eqHI5}
&
I_{11}(t,\varepsilon):=|e^{A_{-}t}P_{-}\int_{0}^{t}\hat{f}_{\varepsilon}(s)ds|\le
 \|e^{A_{-}t}P_{-}\||\int_{0}^{t}\hat{f}_{\varepsilon}(\sigma)d\sigma|\le
 \mathcal N e^{-\nu t}t\omega(\frac{t}{\varepsilon}),
\end{eqnarray}
\begin{eqnarray}\label{eqHI6}
&
I_{12}(t,\varepsilon):=|\int_{-t}^{0}A_{-}e^{-A_{-}s}P_{-}\int_{t}^{t+s}\hat{f}_{\varepsilon}(\sigma)d\sigma
ds|\le \nonumber \\ &
\int_{-t}^{0}\|A_{-}e^{-A_{-}s}P_{-}\||\int_{t}^{t+s}\hat{f}_{\varepsilon}(\sigma)d\sigma|\le
\int_{-t}^{0}\|A_{-}e^{-A_{-}s}P_{-}\||s|\omega(\frac{|s|}{\varepsilon})ds
= \nonumber
\\ & \int_{0}^{t}\|A_{-}e^{A_{-}s}P_{-}\|s\omega(\frac{s}{\varepsilon})ds
\le \int_{0}^{t}\mathcal N \|A\|e^{-\nu
s}s\omega(\frac{s}{\varepsilon})ds=\hat{\mathcal N}
\int_{0}^{t}e^{-\nu s}s\omega(\frac{s}{\varepsilon})ds,
\end{eqnarray}
where $\hat{\mathcal N}:=\mathcal N \|A\|$. From (\ref{eqHI5}) and
(\ref{eqHI6}) we obtain
\begin{equation}\label{eqHI6.1}
I_{1}(t,\varepsilon)\le I_{11}(t,\varepsilon) +
I_{12}(t,\varepsilon)\le \mathcal N e^{-\nu
t}t\omega(\frac{t}{\varepsilon}) + \hat{\mathcal N}
\int_{0}^{t}e^{-\nu s}s\omega(\frac{s}{\varepsilon})ds \nonumber
\end{equation}
and, consequently,
\begin{equation}\label{eqH16.2}
\sup\limits_{t\in\mathbb R_{+}}I_{1}(t,\varepsilon)\le \mathcal N
\sup\limits_{t\in\mathbb R_{+}}e^{-\nu
t}t\omega(\frac{t}{\varepsilon}) + \hat{\mathcal N}
\sup\limits_{t\in\mathbb R_{+}}\int_{0}^{t}e^{-\nu
s}s\omega(\frac{s}{\varepsilon})ds .
\end{equation}
According to Lemma \ref{lG_0} from (\ref{eqH16.2}) we get
\begin{equation}\label{eqH16.3}
\lim\limits_{\varepsilon \to 0^{+}}\sup\limits_{t\in\mathbb
R_{+}}I_{1}(t,\varepsilon)=0 .
\end{equation}

Now we will estimate the integral
\begin{equation}\label{eqHI12.1}
I_{2}(t,\varepsilon):=|-\int_{t}^{+\infty}e^{A_{+}(t-\tau)}P_{+}\hat{f}_{\varepsilon}(\tau)d\tau
|=|\int_{0}^{+\infty}e^{-A_{+}s}P_{+}\hat{f}_{\varepsilon}(t+s)ds|.\nonumber
\end{equation}

Integrating by parts in $s$ we have
\begin{eqnarray}\label{eqHI13}
& \int_{0}^{+\infty}e^{-A_{+}s}P_{+}\hat{f}_{\varepsilon}(t+s)ds
=\nonumber
\\
&
e^{-A_{+}s}P_{+}\int_{t}^{t+s}\hat{f}_{\varepsilon}(\sigma)d\sigma\Big{|}_{0}^{+\infty}+\int_{0}^{+\infty}A_{+}e^{-A_{+}s}P_{+}\int_{t}^{t+s}
\hat{f}_{\varepsilon}(\sigma)d\sigma ds .\nonumber
\end{eqnarray}

Taking into consideration (\ref{eqH14.1})
\begin{equation}\label{eqH13.1}
|\int_{t}^{t+s}\hat{f}_{\varepsilon}(\sigma)d\sigma|\le A s
\nonumber
\end{equation}
we obtain
\begin{equation}\label{qHI14}
|\int_{0}^{+\infty}e^{-A_{+}s}P_{+}\hat{f}_{\varepsilon}(t+s)ds|\le
A\mathcal N e^{-\nu s}s \nonumber
\end{equation}
for any $t,s\in \mathbb R_{+}$ and, consequently,
\begin{equation}\label{eqHI15}
\lim\limits_{s\to +\infty}\sup\limits_{t\ge
0}|\int_{0}^{+\infty}e^{-A_{+}s}P_{+}\hat{f}_{\varepsilon}(t+s)ds|=0
.\nonumber
\end{equation}
Thus we have
\begin{equation}\label{eqHI16}
I_{2}(t,\varepsilon)=|\int_{0}^{+\infty}A_{+}e^{-A_{+}s}P_{+}\int_{t}^{t+s}\hat{f}_{\varepsilon}(\sigma)d\sigma
ds| .\nonumber
\end{equation}
Since
\begin{eqnarray}\label{eqHI17}
&
|\int_{0}^{+\infty}A_{+}e^{-A_{+}s}P_{+}\int_{t}^{t+s}\hat{f}_{\varepsilon}(\sigma)d\sigma
ds|\le \nonumber \\
&
\int_{0}^{+\infty}\|A_{+}e^{-A_{+}s}P_{+}\||\int_{t}^{t+s}\hat{f}_{\varepsilon}(\sigma)d\sigma|\le
\\ & \int_{0}^{+\infty}\mathcal N s^{-1}e^{-\nu
s}s\omega(\frac{s}{\varepsilon})ds =  \mathcal
N\int_{0}^{+\infty}e^{-\nu s}\omega(\frac{s}{\varepsilon})ds
\nonumber
\end{eqnarray}
for any $t\ge 0$, then
\begin{eqnarray}\label{eqHI18.1}
& I_{2}(t,\varepsilon)=\sup\limits_{t\ge
0}|\int_{0}^{+\infty}A_{+}e^{-A_{+}s}P_{+}\int_{t}^{t+s}\hat{f}_{\varepsilon}(\sigma)d\sigma
ds|\le \nonumber \\
& \mathcal N \int_{0}^{+\infty}e^{-\nu
s}\omega(\frac{s}{\varepsilon}ds\to 0 \ \ \mbox{as}\ \varepsilon
\to 0
\end{eqnarray}
uniformly with respect to $t\in \mathbb R_{+}$.

From (\ref{eqA7}), (\ref{eqH16.3}) and (\ref{eqHI18.1}) we have
\begin{equation}\label{eqGS17}
\lim\limits_{\varepsilon \to 0^{+}}\|\varphi_{\varepsilon}\|=0 .
\end{equation}
Lemma is completely proved.
\end{proof}

\begin{theorem}\label{thLG1} Assume that the following conditions
are fulfilled:
\begin{enumerate}
\item $-\mathcal A$ is a sectorial hyperbolic operator;
\item
$f\in C_{b}(\mathbb R_{+},\mathfrak B)$; \item there exists an
element $\bar{f}\in \mathfrak B$ such that
\begin{equation}\label{eqG1.1}
\bar{f}=\lim\limits_{T\to +\infty}\frac{1}{T}\int_{t}^{t+T}f(s)ds
\nonumber
\end{equation}
uniformly with respect to $t\in\mathbb R_{+}$.
\end{enumerate}

Then
\begin{enumerate}
\item[1.] for any $\varepsilon \in (0,\varepsilon_0]$ there exists
a unique solution $\varphi_{\varepsilon}\in C_{b}(\mathbb
R_{+},\mathfrak B)$ of equation (\ref{eqALN1}) compatible in the
limit with $\mathcal P\varphi (0)=0$; \item[2.] for any
$\varepsilon \in (0,\varepsilon_0]$ there exists a unique solution
$\psi_{\varepsilon}\in C_{b}(\mathbb R_{+},\mathfrak B)$ of
equation
\begin{equation}\label{eqLNA1}
x'(\tau)=\mathcal A x(\tau)+f_{\varepsilon}(\tau)),\ \
(f_{\varepsilon}(\tau):=f(\frac{\tau}{\varepsilon}),\ \tau\in
\mathbb R_{+})
\end{equation}
with $\mathcal P\psi_{\varepsilon}(0)=0$; \item[3.]
\begin{equation}\label{eqG1.2}
\lim\limits_{\varepsilon \to
0^{+}}\|\psi_{\varepsilon}-\bar{\psi}\|=0,\nonumber
\end{equation}
where $\bar{\psi}$ is a unique stationary solution of equation
\begin{equation}\label{eqLNA2}
y'(\tau)=\mathcal A y(\tau) +\bar{f} ;
\end{equation}
\item[4.]
\begin{equation}\label{eqG1.2.1}
\lim\limits_{\varepsilon \to 0^{+}}\sup\limits_{t\in\mathbb
R_{+}}|\varphi_{\varepsilon}(\frac{t}{\varepsilon})-\bar{\psi}|=0
.
\end{equation}
\end{enumerate}
\end{theorem}
\begin{proof} By Theorem \ref{thLS11_2} equation (\ref{eqALN1})
(respectively, equation (\ref{eqLNA1}) has a unique solution
$\varphi_{\varepsilon}\in C_{b}(\mathbb R_{+},\mathfrak B)$
(respectively, $\psi_{\varepsilon}\in C_{b}(\mathbb
R_{+},\mathfrak B)$) compatible in the limit with $\mathcal P
\varphi_{\varepsilon}(0)=0$ (respectively, $\mathcal
P\psi_{\varepsilon}(0)=0$).

Under the conditions of Theorem \ref{thLG1} equation
(\ref{eqLNA2}) has a unique stationary solution
$\bar{\psi}=-(\mathcal A)^{-1}\bar{f}$. Let us make a change of
the variable $x=\bar{\psi}+y$ in equation (\ref{eqLNA2}), then we
obtain
\begin{equation}\label{eqLNA3}
y'(\tau)=\mathcal A y(\tau)+ \hat{f}_{\varepsilon}(\tau),
\end{equation}
where $\hat{f}_{\varepsilon}(\tau):=f_{\varepsilon}(\tau)-\bar{f}$
for any $\tau\in \mathbb R_{+}$. Note that
$\psi_{\varepsilon}-\bar{\psi}\in C_{b}(\mathbb R_{+},\mathfrak
B)$ is a solution of equation (\ref{eqLNA3}) and $\mathcal
P(\psi_{\varepsilon}(0)-\bar{\psi})=0$. By Theorem \ref{thLS11_2}
we have
\begin{equation}\label{eqLNA4}
\psi_{\varepsilon}(\tau)-\bar{\psi}=\int_{0}^{+\infty}G(\tau
-s)\hat{f}_{\varepsilon}(s)ds .\nonumber
\end{equation}
According to Lemma \ref{lG1} we get
\begin{equation}\label{eqLN4.1}
\lim\limits_{\varepsilon\to
0^{+}}\|\psi_{\varepsilon}-\bar{\psi}\|=0 .\nonumber
\end{equation}
Since
$\psi_{\varepsilon}(\tau)=\varphi_{\varepsilon}(\frac{\tau}{\varepsilon})$
for any $\tau \in \mathbb R_{+}$ and $\varepsilon \in (0,\alpha]$,
we get (\ref{eqG1.2.1}). Theorem is proved.
\end{proof}

\begin{coro}\label{cor L10}
Under the conditions of \textsl{Theorem} \ref{thLG1} the following
statements hold:
\begin{enumerate}
\item if the function $f\in$ $ C_{b}(\mathbb R_{+},$ $\mathfrak
B)$ is asymptotically stationary (respectively, asymptotically
$\tau$-periodic, asymptotically quasi-periodic with the spectrum
of frequencies $\nu_1,\ldots,\nu_k$, asymptotically Bohr almost
periodic, asymptotically Bohr almost automorphic, asymptotically
Birkhoff recurrent, positively Lagrange stable), then equation
(\ref{eqLNA1}) has a unique solution $\varphi_{\varepsilon} \in
C_{b}(\mathbb R_{+},\mathfrak B)$ with $\mathcal P
\varphi_{\varepsilon}(0)=0$ which is asymptotically stationary
(respectively, asymptotically $\tau$-periodic, asymptotically
quasi-periodic with the spectrum of frequencies
$\nu_1,\ldots,\nu_k$, asymptotically Bohr almost periodic,
asymptotically Bohr almost automorphic, asymptotically Birkhoff
recurrent, positively Lagrange stable); \item
\begin{equation}\label{eqG1.2.01}
\lim\limits_{\varepsilon \to 0^{+}}\sup\limits_{t\in\mathbb
R_{+}}|\varphi_{\varepsilon}(\frac{t}{\varepsilon})-\bar{\psi}|=0
,\nonumber
\end{equation}
where $\bar{\psi}$ is a unique stationary solution of equation
(\ref{eqLNA2}).
\end{enumerate}
\end{coro}
\begin{proof}
This statement follows from Theorems \ref{th2} and \ref{thLG1},
because $\mathfrak L_{\varepsilon f}=\mathfrak L_{f}$ for any
$\varepsilon \in (0,\alpha]$.
\end{proof}

\subsection{Semi-Linear Equations}\label{S6.2}

Consider the following differential equation
\begin{equation}\label{eqG0}
x'(t)=\mathcal A x(t)+f(t)+F(t,x(t)),
\end{equation}
where $f\in C(\mathbb R_{+},\mathfrak B)$, $F\in C(\mathbb
R_{+}\times \mathfrak B,\mathfrak B)$ and $\mathcal A: D(\mathcal
A)\to \mathfrak B$ is a linear operator acting from $D(\mathcal
A)\subseteq \mathfrak B$ to $\mathfrak B$.

In this section we will consider a differential equation
(\ref{eqG0}) when the linear operator $\mathcal A$ is an
infinitesimal operator which generates a $C_0$-semigroup
$\{U(t)\}_{t\ge 0}$.

Below we will use the following conditions:
\begin{enumerate}
\item[\textbf{(G1)}:] $F(t,0)=0$ for any $t\ge 0$;
\item[\textbf{(G2)}:] there exists a positive constant $L$ such
that
\begin{equation}\label{eqGS2}
|F(t,x_{1})-F(t,x_{2})|\le L|x_{1}-x_{2}|\nonumber
\end{equation}
for any $x_1,x_2\in \mathfrak B$ and $t\in\mathbb R_{+}$;
\item[\textbf{(G3)}:] there exists functions $\bar{F}\in
C(\mathfrak B,\mathfrak B)$ and $\omega :\mathbb R_{+}\times
\mathbb R_{+}\to \mathbb R_{+}$ (respectively, an element
$\bar{f}\in \mathfrak B$ and function $\omega\in \Psi$) such that
\begin{equation}\label{eqGS3}
\frac{1}{T}\Big{|}\int_{t}^{t+T}[F(s,x)-\bar{F}(x)]dt\Big{|}\le
\omega(T,r)\nonumber
\end{equation}
(respectively,
\begin{equation}\label{eqGS3.1}
\frac{1}{T}\Big{|}\int_{t}^{t+T}[f(s)-\bar{f}]dt\Big{|}\le
\omega(T)\nonumber\ \ )
\end{equation}
and $\omega(\cdot,r)\in \Psi$ for any $t\in \mathbb R_{+}$, $T>0$,
$r>0$ and $x\in B[0,r]$.
\end{enumerate}

\begin{remark}\label{remC1} If the function $F$ satisfies
condition (\textbf{G2}) then its averaging $\bar{F}$ also
satisfies (\textbf{G2}) with the same constant $L$.
\end{remark}

The standard form of (\ref{eqG0}) is
\begin{equation}\label{eqG1}
x'(t)=\varepsilon (\mathcal{A}x(t)+f(t)+F(t,x(t))) .
\end{equation}
We will consider also the following equations
\begin{equation}\label{eqG2}
x'(t)=\mathcal A
x(t)+f(\frac{t}{\varepsilon})+F(\frac{t}{\varepsilon},x(t))),
\end{equation}
where $f_{\varepsilon}(t):=f(\frac{t}{\varepsilon})$
(respectively, $F_{\varepsilon}(t,x):=F(\frac{t}{\varepsilon},x)$)
for any $t\in\mathbb R_{+}$ (respectively, for any $(t,x)\in
\mathbb R_{+}\times \mathfrak B$), $\varepsilon
\in(0,\varepsilon_0]$ and $\varepsilon_{0}$ is some fixed small
positive number. Along with equations (\ref{eqG1})-(\ref{eqG2}) we
will consider the following averaged differential equation
\begin{equation}\label{eqG5}
x'(t)=\mathcal A x(t)+ \bar{f} + \bar{F}(x(t)).
\end{equation}

\begin{theorem}\label{thGS}
Suppose that the following conditions hold:
\begin{enumerate}
\item $-\mathcal A$ is a sectorial hyperbolic operator;  \item the
function $F$ satisfies conditions (\textbf{G1})-(\textbf{G3});
\item the functions $f$ and $F$ are Lagrange stable; \item
\begin{equation}\label{eqGS.1}
L<\frac{\nu}{2\mathcal N }\ ,\nonumber
\end{equation}
where $\mathcal N$ and $\nu$ there are the numbers figuring in the
Definition \ref{defH1}.
\end{enumerate}

Then there exists a positive number $\varepsilon_0$ such that for
any $0<\varepsilon \le \varepsilon_{0}$
\begin{enumerate}
\item equation (\ref{eqG2}) has a unique solution
$\psi_{\varepsilon}\in C_{b}(\mathbb R_{+},\mathfrak B)$ with
$\mathcal P\psi_{\varepsilon}(0)=0$; \item the solution
$\psi_{\varepsilon}$ of equation (\ref{eqG2}) is compatible in the
limit, i.e., $\mathfrak L_{F_{Q_{\varepsilon}}}\subseteq \mathfrak
L_{\psi_{\varepsilon}}$, where
$Q_{\varepsilon}:=\overline{\psi_{\varepsilon}(\mathbb R_{+})}$;
\item
\begin{equation}\label{eqG6}
\lim\limits_{\varepsilon \to 0}\sup\limits_{t\in\mathbb
R_{+}}|\psi_{\varepsilon}(t)-\bar{\psi}|=0,\nonumber
\end{equation}
where $\bar{\psi}$ is a unique stationary solution of equation
(\ref{eqG5}).
\end{enumerate}
\end{theorem}
\begin{proof}
According to conditions (\textbf{G1})-(\textbf{G2}) we have
\begin{equation}\label{eqL1}
|F(t,x)|\le  Lr\ \ \mbox{and}\ \ L\le \alpha <\frac{\nu}{2\mathcal
N}\nonumber
\end{equation}
for any $t\ge 0$, $|x|\le r$ and $r\ge 0$. According to Theorems
\ref{th4A} and \ref{thLS01} equation (\ref{eqG2}) has a unique
compatible in the limit solution $\psi_{\varepsilon}$ from
$C_{b}(\mathbb R_{+},\mathfrak B)$ with $\mathcal P
\psi_{\varepsilon}(0)=0$.

Let $\bar{\psi}$ be a unique stationary solution of equation
(\ref{eqG5}) which there exists under the conditions of Theorem
(see Theorem \ref{th4A} and Remark \ref{remC1}). Note that
$\psi_{\varepsilon}-\bar{\psi}\in C_{b}(\mathbb R_{+},\mathfrak
B)$ is a unique solution of equation
\begin{eqnarray}\label{eqL1.1}
& u'=\mathcal A u +f_{\varepsilon}(t)-\bar{f}
+F_{\varepsilon}(t,\psi_{\varepsilon}(t))-\bar{F}(\bar{\psi})=\nonumber
\\
& \mathcal A u
+[f_{\varepsilon}(t)-\bar{f}]+[F_{\varepsilon}(t,\psi_{\varepsilon}(t))-F_{\varepsilon}(t,\bar{\psi})]+[F_{\varepsilon}(t,\bar{\psi})-\bar{F}(\bar{\psi})]=\nonumber
\\
& \mathcal A u +
\hat{f}_{\varepsilon}(t)+\hat{F}_{\varepsilon}(t,\bar{\psi})+[F_{\varepsilon}(t,\psi_{\varepsilon}(t))-F_{\varepsilon}(t,\bar{\psi})]\nonumber
\end{eqnarray}
with $\mathcal P(\psi_{\varepsilon}(0)-\bar{\psi})=0$ and,
consequently,
\begin{equation}\label{eqL1.01}
\psi_{\varepsilon}(t)-\bar{\psi}=\int_{0}^{+\infty}G(t-\tau)\big{(}\hat{f}_{\varepsilon}(\tau)+
\hat{F}_{\varepsilon}(\tau,\bar{\psi})+[F_{\varepsilon}(\tau,\psi_{\varepsilon}(\tau))-F_{\varepsilon}(\tau,\bar{\psi})]\big{)}d\tau
.
\end{equation}
From (\ref{eqL1.01}) we get
\begin{equation}\label{eqL1.3}
|\psi_{\varepsilon}(t)-\bar{\psi}|\le
J_{1}(t,\varepsilon)+J_{2}(t,\varepsilon)+J_{3}(t,\varepsilon)
\end{equation}
for any $t\in\mathbb R_{+}$ and $\varepsilon \in (0,\alpha]$,
where
\begin{eqnarray}\label{eqL1.4}
&
J_{1}(t,\varepsilon):=|\int_{0}^{+\infty}G(t-\tau)\hat{f}_{\varepsilon}(\tau)d\tau|,
\nonumber \\
&
J_{2}(t,\varepsilon):=|\int_{0}^{+\infty}G(t-\tau)\hat{F}_{\varepsilon}(\tau,\bar{\psi})d\tau|\
\ \mbox{and}\nonumber
\\
&
J_{3}(t,\varepsilon):=|\int_{0}^{+\infty}G(t-\tau)[F_{\varepsilon}(t,\psi_{\varepsilon}(t))-F_{\varepsilon}(t,\bar{\psi})]\big{)}d\tau|\nonumber
.
\end{eqnarray}
By Lemma \ref{lG1} we have
\begin{equation}\label{eqL1.5}
\lim\limits_{\varepsilon\to )^{+}}\sup\limits_{t\in\mathbb
R_{+}}J_{i}(t,\varepsilon)=0\ \ (i=1,2).
\end{equation}
Taking in consideration (\ref{eqHG2}) we can estimate
$J_{3}(t,\varepsilon)$ as follows
\begin{eqnarray}\label{eqL1.6}
&
J_{3}(t,\varepsilon)=|\int_{0}^{+\infty}G(t-\tau)[F_{\varepsilon}(t,\psi_{\varepsilon}(t))-F_{\varepsilon}(t,\bar{\psi})]\big{)}d\tau|\le
\nonumber \\
&
L\|\psi_{\varepsilon}-\bar{\psi}\|\int_{0}^{+\infty}\|G(t-\tau)\|d\tau\le
\frac{2\mathcal N L}{\nu}\|\psi_{\varepsilon}-\bar{\psi}\| .
\end{eqnarray}
From (\ref{eqL1.01}), (\ref{eqL1.3}) and (\ref{eqL1.6}) we obtain
\begin{equation}\label{eqL1.7}
\|\psi_{\varepsilon}-\bar{\psi}\|\le \frac{\nu
B(\varepsilon)}{\nu-2\mathcal N L}
\end{equation}
for any $\varepsilon \in (0,\alpha]$, where
\begin{equation}\label{eqL1.8}
B(\varepsilon):=\sup\limits_{t\in\mathbb
R_{+}}J_{1}(t,\varepsilon) + \sup\limits_{t\in\mathbb
R_{+}}J_{2}(t,\varepsilon) .
\end{equation}
From (\ref{eqL1.5}), (\ref{eqL1.7}) and (\ref{eqL1.8}) we have
\begin{equation}\label{eqL1.9}
\lim\limits_{\varepsilon \to
0^{+}}\|\psi_{\varepsilon}-\bar{\psi}\|=0 .\nonumber
\end{equation}
Theorem is completely proved.
\end{proof}

\begin{remark}\label{remGR1} Note that Theorem \ref{thGS} remains
true for the equation
$$
x'=\varepsilon (\mathcal A x +f(t)+F(t,x,\varepsilon))
$$
if the function $F(t,x,\varepsilon)$ is Lipschitzian and its
Lipschitz constant $L(\varepsilon)$ satisfies to the following
condition
\begin{equation}\label{eqGR0}
\lim\limits_{\varepsilon \to 0^{+}}L(\varepsilon)=0 .
\end{equation}
\end{remark}

\begin{example}\rm
Consider the heat equation with small parameter $\varepsilon$  on
the interval [0,1] with Dirichlet boundary condition:
\begin{equation}\label{eqEx1}
\frac{\partial u}{\partial t} = \varepsilon
\Big{[}\frac{\partial^2 u}{\partial \xi^2}+ \frac{1}{3}(\sin
t+\cos\sqrt{2}t +e^{-t})\sin u\Big{]}
\end{equation}
on the interval [0,1] with Dirichlet boundary condition
$u(t,0)=u(t,1)=0,\;\;\; t>0$.

Let $A$ be the operator defined by $A\varphi (x)= \varphi^{''}(x)$
($0<x<1$), then $A: D(A)=H^2(0,1) \cap H^1_0 (0,1)\to L^2(0,1)$
(for more details see \cite[Ch.I]{Hen_1981}). Denote $H :=
L^2(0,1)$ and the norm on $H$ by $||\cdot||$. Then equation
(\ref{eqEx1}) can be written as an abstract evolution equation
\begin{equation}\label{aeex}
y'(t) =\mathcal{A}  y(t)+ {F}(t,y(t))\nonumber
\end{equation}
on the Hilbert space $H$, where
\begin{equation}
y(t) := u(t,\cdot), \quad {F}(t,y(t)):= f(t, u(t,\cdot)).\nonumber
\end{equation}
Note that, the operator $-A$ is sectorial \cite[Ch.I]{Hen_1981},
$\sigma(A)=\{-n^2\pi^2|\ n\in\mathbb N\}$ (i.e.,
$\sigma_{-}(A)=\sigma(A)$, $\sigma_{+}(A)=\emptyset$ and hence
$P_{-}=Id_{H}$ and $P_{+}=0$) and generates a $\mathcal
C^0$-semigroup $\{U(t)\}_{t\ge 0}=\{e^{At}\}_{t\ge 0}$ on $H$.
$-A$ is a sectorial \cite[Ch.I]{Hen_1981} hyperbolic operator,
$\sigma_{+}(-A)=\sigma(-A)=\{n^2\pi^2|\ n\in\mathbb N\}$ and
$\sigma_{-}(-A)=\emptyset$. Note that $Lip(F)\le 1$, so it is
immediate to verify that conditions
(\textbf{\textbf{G1}})-(\textbf{G2}) hold. It easy to check
condition (\textbf{G3}). Finally note that $F$ is asymptotically
quasi-periodic in $t$, uniformly w.r.t. $y$ on every compact
subset from $H$. By Theorem \ref{thGS} equation
\begin{equation}\label{eqEx2}
\frac{\partial u}{\partial t} = \frac{\partial^2 u}{\partial
\xi^2}+ \frac{1}{3}\Big{(}\sin (t/\varepsilon) +\cos(\sqrt{2}
t/\varepsilon) +e^{-t/\varepsilon}\Big{)}\sin u\nonumber
\end{equation}
has a unique asymptotically almost periodic solution
$\psi_{\varepsilon}$ with
$P_{-}\psi_{\varepsilon}(0)=\psi_{\varepsilon}(0)=0$ such that
$$\lim\limits_{\varepsilon \to 0}\sup\limits_{t\ge
0}|\psi_{\varepsilon}(t)|=0.$$
\end{example}


\section{Funding}

This research was supported by the State Program of the Republic
of Moldova "Multivalued dynamical systems, singular perturbations,
integral operators and non-associative algebraic structures
(20.80009.5007.25)".

\section{Conflict of Interest}

The author declare that he does not have conflict of interest.

\end{document}